\documentclass[3p,times]{elsarticle}

\usepackage{amssymb}
\usepackage{amsmath}
\usepackage{amsthm}

\usepackage{multicol}
\usepackage{tikz,tkz-tab}              
\usetikzlibrary{cd}                
\usetikzlibrary{arrows.meta, decorations.pathmorphing}
\usetikzlibrary{calc}
\usetikzlibrary{positioning}
\usepackage{verbatim}
\usepackage{quiver}
\usepackage{cutwin} 
\usepackage{enumerate}
\usepackage{stmaryrd} 
\SetSymbolFont{stmry}{bold}{U}{stmry}{m}{n}

\newtheorem{theorem}{Theorem}[section]
\newtheorem{lemma}[theorem]{Lemma}
\newtheorem{proposition}[theorem]{Proposition}
\newtheorem{corollary}[theorem]{Corollary}

\theoremstyle{definition}
\newtheorem{definition}[theorem]{Definition}
\newtheorem{example}[theorem]{Example}

\theoremstyle{remark}
\newtheorem{remark}[theorem]{Remark}

\numberwithin{equation}{section}

\newcommand{\Q}{\mathbb{Q}}
\newcommand{\GL}{\mathrm{GL}}
\newcommand{\Tr}{\mathrm{Tr}}

\newcommand{\toa}{\overset{a}{\to}}
\newcommand{\toas}{\overset{a}{\to}\mkern-8mu^*}
\newcommand{\tob}{\overset{b}{\to}}
\newcommand{\tobs}{\overset{b}{\to}\mkern-8mu^*}
\newcommand{\eqab}{\overset{ab}{\sim}}

\newcommand{\Sym}{\mathfrak{S}}
\newcommand{\braid}{\boldsymbol{\sigma}}
\newcommand{\orderings}{\mathrm{O}}
\newcommand{\Acyc}{\mathrm{Acyc}}
\newcommand{\Red}{\mathrm{Red}}
\newcommand{\Aut}{\mathrm{Aut}}
\newcommand{\Inn}{\mathrm{Inn}}
\newcommand{\Out}{\mathrm{Out}}
\newcommand{\Gal}{\mathrm{Gal}}

\newcommand{\blue}[1]{\textcolor{blue}{#1}}
\newcommand{\red}[1]{\textcolor{red}{#1}}
\newcommand{\purple}[1]{\textcolor{purple}{#1}}

\newcommand{\domain}{\mathcal{D}}
\newcommand{\orbit}[1]{\mathcal{O}_{#1}}
\newcommand{\C}{\mathcal{C}} 
\newcommand{\cgs}{\mathcal{CGS}} 

\newcommand{\defit}[1]{\emph{#1}}
\newcommand{\pres}[2]{\left\langle #1 \middle\vert #2 \right\rangle}
\newcommand{\R}{\mathbb{R}}
\newcommand{\longline}{\scalebox{2}[1]{$-$}}

\newcommand{\one}{{1\!\!1}} 
\newcounter{intege}
\newcommand{\polygon}[5]{ 
  \foreach \t in {1,...,#3} {
    \coordinate (#2\t) at ($#1+(90-\t*360/#3:#4)$);
  }
  \draw[thin,black,fill=white,opacity=0.3,densely dashed] #1 circle (#4);
  \setcounter{intege}{1}
  \pgfmathsetcounter{intege}{1}
  \foreach \object in {#5}{
    \ifthenelse{\not\equal{\object}{}}{
      \filldraw[black] ($#1+($(90-\theintege*360/#3:#4)$)$) circle(2pt);
    }{
      \filldraw[black] ($#1+($(90-\theintege*360/#3:#4)$)$) circle(2pt);
    }
    \node[inner sep=0pt] at ($#1+($1.12*(90-\theintege*360/#3:#4)$)$) {$\scriptstyle\object$};
    \pgfmathsetcounter{intege}{\theintege+1}
    \setcounter{intege}{\theintege}
  }
}

\journal{Journal of Algebra}

\begin{document}

\begin{frontmatter}



\title{The restricted Hurwitz action on cameral decompositions of Coxeter elements} 

\author{Emir Melliti
} 
\ead{emir.melliti@univ-grenoble-alpes.fr}
\affiliation{organization={Univ. Grenoble Alpes, CNRS, IF}, 
            postcode={38000}, 
            city={Grenoble},
            country={France}}

\begin{abstract}
Given a dual Coxeter system $(W,T)$ of rank $n$, the Hurwitz action of the braid group on the reduced $T$-decompositions of a Coxeter element is transitive. When we choose a geometric representation of $W$, the reflections with respect to the $n$ walls of one given chamber form a generating set $S$ that we call cameral. An ordering of $S$ is called a cameral decomposition and the product a cameral Coxeter element, which depends only on the quiver obtained by orienting the Coxeter diagram with vertices $S$ according to the ordering. We show that the reduced decompositions of a cameral Coxeter elements for which the underlying set $S$ satisfies that $(W,S)$ is a Coxeter system, are necessarily cameral, and we count the number of chambers from which a cameral Coxeter element can be obtained. We show that one can go through all these chambers using the restricted Hurwitz action on the set of cameral decompositions, which is generated by transpositions of adjacent commuting reflections and by the Garside element $\delta$ of the dual Garside structure, which acts by a quiver mutation of the orientation of the Coxeter diagram.
\end{abstract}

\begin{keyword}
reflection groups \sep Coxeter groups \sep Coxeter elements \sep Hurwitz action

\MSC[2020] 20F55 \sep 20F36

\end{keyword}

\end{frontmatter}


\section{Introduction}
    \subsection{Context}
        The dual approach to braid groups was initiated Birman, Ko, and Lee in \cite{birmanNewApproachWord1998}, before it was generalized to spherical Artin-Tits groups by Bessis in \cite{bessisDualBraidMonoid2003a} and independently by Brady and Watt \cite{bradyPartialOrderSymmetric2001, brady$Kpi1$sArtinGroups2002}. Given a spherical Coxeter system $(W,S)$ of rank $n$, the set of reflections $T$, i.e. the closure of $S$ under conjugation, is finite, and one can take it as a new set of generators. This set of generators allows one to define a dual braid monoid, $M'$, which is an alternative braid monoid embedding in the same Artin-Tits group $A$ of spherical type, providing an alternative Garside structure for this group. In the modern language of Garside theory developed by Dehornoy, Digne, Godelle, Krammer and Michel in \cite{dehornoyFoundationsGarsideTheory2015a} this Garside structure is derived from a short right-complemented presentation $\pres{T}{\mathcal{R}'}^+$ of the dual monoid $M'$, that is a presentation where every pair $(s\ne t)\in T^2$ is given exactly one right-common multiple of $T$-length $2$ in $M'$ : $st'= ts'$ with $s',t' \in T$. There are two natural choices for defining these right complements, $s'=tst$ and $t'=t$ or $s'=s$ and $t'=sts$, that corresponds to two relations in $W$, namely $st=t(tst)$ and $ts=s(sts)$. To choose between the two, one can fix a partial order on the set $T$ that depends on the choice of a Coxeter element $c$ in $W$, that is, the product in any order of the elements of a subset $S'\subset T$ for which $(W,S')$ is a Coxeter system in the classical sense (note that there are other non-equivalent definitions of a Coxeter element in the literature,  we clarify this in Section \ref{sec:chamber_coxeter}). These relations are called dual braid relations, and they define the dual braid monoid which embeds in the Artin-Tits group like the classical Artin-Tits monoid, like any Garside monoid. Note that without the Garside structure, i.e. outside spherical type, the classical Artin-Tits monoid embeds in the Artin-Tits group \cite{parisArtinMonoidsInject2002}. The dual monoid can be defined in general type with an infinite $T$, but we don't know if it embeds in the Artin-Tits group. In classical Coxeter theory, adding the relations $s^2=1$ for all $s\in S$ surjects the Artin-Tits monoid onto the Coxeter group $W$, and the Garside element $\Delta$ is sent to the longest element in $W$. In the dual theory, adding the relation $t^2=1$ for all $t \in T$ also surjects the dual braid monoid onto $W$, but the Garside element is sent to the chosen Coxeter element $c$.
        \[\begin{tikzcd}
            & A & \\
            {\pres{W}{S}^+} && {\pres{W}{T}^+} \\
            & W
            \arrow[hook, from=2-1, to=1-2]
            \arrow["{s^2=1}"'{pos=0.2}, two heads, from=2-1, to=3-2]
            \arrow[hook, from=2-3, to=1-2]
            \arrow["{t^2=1}"{pos=0.2}, from=2-3, to=3-2]
        \end{tikzcd}\]
        \begin{example}
            We illustrate the dual setup in type $A_2$. Let $S = {s, t}$. The monoid presentation $\pres{S}{sts=tst}^+$ defines the Artin-Tits monoid, which embeds in the Artin-Tits group with presentation $\pres{S}{sts=tst}$. The monoid has a Garside structure with Garside element $\Delta = sts=tst$. Adding the relations $s^2=1$ and $t^2=1$ gives the Coxeter group of type $A_2$, which is the symmetric group $\Sym_3$. The set of reflections is $T = \{s, t, u=sts\}$, and the dual braid monoid is defined by the presentation $\pres{T}{st=tu=us}^+$. It is also a Garside monoid with Garside element $\Delta = st=tu=us$, which is sent to the rotation $st$ in the Coxeter group. 
            \[
            \begin{tikzpicture}[
                scale=1.5,
                every node/.style={font=\small}
            ]

            \def\R{1}
            \def\ext{1.15} 

            \foreach \i in {0,...,5} {
                \coordinate (P\i) at (\i*60:\R);
            }

            \draw[thick] (P0) -- (P1) -- (P2) -- (P3) -- (P4) -- (P5) -- cycle;
            \foreach \i in {0,...,5} {
                \fill (P\i) circle (1.5pt);
            }

            \draw[thick] ($(P3)!\ext!(P0)$) -- ($(P0)!\ext!(P3)$);
            \draw[thick] ($(P4)!\ext!(P1)$) -- ($(P1)!\ext!(P4)$);
            \draw[thick] ($(P5)!\ext!(P2)$) -- ($(P2)!\ext!(P5)$);

            \node at ($(P0)!0.25!(P3)$) [above] {$s$};
            \node at ($(P1)!0.25!(P4)$) [left] {$t$};
            \node at ($(P2)!0.25!(P5)$) [below left] {$u$};

            \draw[-{Latex[length=3mm]}, thick, blue]
                (0:\R*\ext) arc (0:120:\R*\ext);
            \node[blue] at (90:\R*\ext+0.1) {$c=st$};

            \end{tikzpicture}
            \]
        \end{example}
        One fundamental combinatorial property of Coxeter systems is the so-called Matsumoto's lemma : one can pass from any reduced $S$-decomposition of an element to any other one by successive applications of braid relations only (i.e. without the relations $s^2=1$). The same question for reduced $T$-decompositions and dual braid relations leads to investigate the Hurwitz action of the braid group on reduced decompositions, introduced in 1891 by Hurwitz \cite{hurwitzUeberRiemannscheFlachen1891}. The crossing $\braid_i$ acts by sending the decomposition $t_1\ldots t_i t_{i+1} \ldots t_k$ to $t_1 \ldots t_{i-1} t_{i+1} (t_{i+1}^{-1}t_i t_{i+1}) t_{i+2} \ldots t_k$, swapping the $i$-th and $(i+1)$-th factors and conjugating the $i$-th factor by the $(i+1)$-th. For the Matsumoto property to hold, the Hurwitz action on reduced $T$-decompositions of an element $w$ should be transitive. It has been known that the transitivity holds when $w$ is a parabolic Coxeter element (note that there are also multiple definitions for these in the literature) and there are some generalizations to some non-spherical types (see \cite{baumeisterNoteTransitiveHurwitz2014}). A necessary and sufficient condition for the Hurwitz action to be transitive is given in \cite{baumeisterHurwitzActionFinite2017}.\par
        Another way of studying Coxeter groups is through its linear representations (see \cite{bourbakiGroupesAlgebresLie2007}). The reflections of $T$ act by orthogonal reflection with respect to a corresponding set of hyperplanes, that delimit \emph{chambers}. When we take the set $S'$ of the $n$ reflections with respect to the walls of a chamber, the pair $(W,S')$ is a Coxeter system, and the product in any order of these reflections is a Coxeter element. Then some questions arise from this way of constructing Coxeter elements.\begin{enumerate}
            \item Can any Coxeter element be obtained this way ?
            \item\label{question:how_many} From how many chambers can a Coxeter element be obtained ?
            \item\label{question:hur_transitive} Can we use the Hurwitz action to go through all these chambers ?
        \end{enumerate}
        The third question is imprecise on purpose, and one goal is to precise it in an interesting way. This paper investigates the combinatorial interactions between the notion of chamber and the Hurwitz action on reduced $T$-decompositions of Coxeter elements with these questions as guidelines.
    \subsection{Results}
        Classical results of Coxeter theory (see e.g. \cite{bourbakiGroupesAlgebresLie2007}) state that two Coxeter elements defined as the product in two different orders of the same set $S'$ are conjugated, and that all sets $S'$ corresponding to the walls of a chamber are also conjugated. We call these generating sets \emph{cameral generating sets} (\emph{cam-sets}), they have the property of being \emph{Coxeter generating sets}, i.e. $(W, S')$ is a Coxeter system. The product in any order of the reflections of a cameral generating set is called a \emph{cameral Coxeter element}, and they form a whole conjugacy class in $W$. It has been pointed out (see for example \cite{reinerNonconjugateCoxeterElements2015}) that there might be multiple conjugacy class of Coxeter elements, thus some Coxeter elements cannot be obtained from cameral generating sets. There is only one conjugacy class of Coxeter element if and only if $W$ is a Weyl group (this includes type $A$). This gives rise to some inverse question : can a cameral Coxeter element be written as the product of the reflections of a Coxeter generating set $S$ that is not cameral ? Our first result is that the answer is no.
        \begin{proposition}[Proposition \ref{prop:cam_cox_cam_dec}]
            Let $W \subset \GL(V)$ be a finite geometric reflection group, and $T$ the set of reflections, and let $c$ be a cameral Coxeter element. Then, all Coxeter decompositions of $c$ are cameral decompositions.
        \end{proposition}
        In order to count the number of chambers with which a cameral Coxeter element can be obtained, we use a simple trick to reduce the work to counting the number of distinct Coxeter elements that can be written from the reflections on the walls of a fixed chamber, which are known to correspond to the acyclic orientations of the Coxeter diagram of which there are $2^{n-1}$ in irreducible type. This gives the following result :
        \begin{corollary}[Corollary \ref{coro:numbers}]
            Let $W$ be a finite irreducible geometric reflection group. For every cam-set $S$, there are exactly $2^{n-1}$ simple Coxeter elements for $S$. For every Coxeter element $c$, there are exactly $h\times 2^{n-2}$ cam-sets for which $c$ is simple, and thus there are $h\times 2^{n-1}$ chambers from which we can obtain $c$.
        \end{corollary}
        In the first question, we were only interested in Coxeter generating sets, but the Hurwitz action on $T$ decomposition can produce a $T$-decomposition of a Coxeter element on a set $S'$ that is not a Coxeter generating set (see Example \ref{exa:hurwitz_preserv_no_cox}). Thus, we need to restrict the Hurwitz action so the support of the produced $T$-decompositions are Coxeter generating sets. We will use a notion of \emph{partial group action} that we take from \cite{kellendonkPartialActionsGroups2004}. A group $G$ (here, $B_n$) acts partially on a set $X$ (here $\Red^c_T(c)$, the set of reduced $T$-decompositions of a cameral Coxeter element $c$ whose underlying set is a cameral generating set) when the action $G\times X \to X$ is a partial map with some composition axioms. Simply restricting a transitive group action to a subset gives a transitive sub-partial action so, this is not a very interesting way to precise the third question. Another natural way of defining a sub-partial action from a group action is by a generating set of allowed pairs $(g,x)$ on which the partial action is defined. Here, we take on one hand the pairs $(\sigma_i,u)$ where $u=(t_1\ldots t_i, t_{i+1} \ldots t_n)$ where $t_i$ and $t_{i+1}$ commutes. These do not change the underlying set of the decomposition, thus it stays in the same chamber. On the other hand, we take all pairs $(\C,u)$ where $\C = \braid_1 \cdots \braid_{n-1}$. Starting with $u=(t_1,\ldots t_n)$ a cameral decomposition, this conjugates all the reflections by $t_n$ and swaps $t_n$ to the first position. This changes the chamber by going through the wall which corresponds to the kernel of $t_n$, and it changes the orientation of the Coxeter diagram by a quiver mutation. Then, the third question can be asked in the following sense. Is the sub-partial action $\rho$ of the Hurwitz action generated by these two types of pairs transitive ?
        The answer is yes, and given by our second result :
        \begin{theorem}[Theorem \ref{thm:always_eqab}]
            Let $W$ be a finite irreducible geometric reflection group of rank $n$. Let $c_0$ be a cameral Coxeter element, and let $u, v$ be two cameral decompositions in $\Red_{T}^c(c_0)$. Then $u \eqab v$.
        \end{theorem}
        Here, $\eqab$ is the equivalence relation associated with the partial Hurwitz action $u \eqab v \Leftrightarrow \exists g, \rho(g,u)=v$.\par
        We obtain this result by counting the size of the orbit of one decomposition under $\rho$ and showing that it matches the number of cameral decompositions of one cameral Coxeter element and is equal to $n! \times h/2$ (where $h$ is the order of a Coxeter element, also called \emph{Coxeter number}).\par
        We start in Section \ref{sec:coxeter_groups} with the definition and basic properties of Coxeter groups. Then we introduce the notions of Coxeter elements and chambers in Section \ref{sec:chamber_coxeter}, by carefully describing and comparing the different definitions that appear in the literature. We define the Hurwitz action in Section 4, and we add in Section 5 some background on partial actions of groups to get a better way of stating our result. In Section 6, we give a positive answer to the three questions in type $A_n$ by investigating in details the action of the symmetric group $W =\Sym_{n+1}$ on $\{1,\ldots,n+1\}$. Finally, in Section 7, we prove the same results in general type.

\section{Artin-Tits, Coxeter groups}\label{sec:coxeter_groups}
    Artin-Tits monoids are defined by their presentation in terms of generators and relations, based on the entry of a Coxeter matrix, often represented as a Coxeter graph. Here, $n$ will denote the number of generators.
    \begin{definition}
        A Coxeter matrix on a set $S$ is a matrix $M$ with integer entries $m_{s,t}$ for $s,t \in S$, such that $m_{s,s} = 1$ for all $s \in S$ and $m_{s,t} = m_{t,s} \geq 2$ for all $s \neq t$ in $S$.
    \end{definition}
    \begin{example}
    Let $n=3$. Here is a Coxeter matrix and the associated Coxeter graph:\\
        \begin{minipage}{0.5\textwidth}
            \[
            \mathcal{M} = \begin{pmatrix}
                1 & 3 & 2 \\
                3 & 1 & 4 \\
                2 & 4 & 1
            \end{pmatrix}
            \]
        \end{minipage}
        \begin{minipage}{0.5\textwidth}
            \[\begin{tikzcd}
                a && b && c
                \arrow[no head, from=1-1, to=1-3]
                \arrow["4", no head, from=1-3, to=1-5]
            \end{tikzcd}\]
        \end{minipage}
    \end{example}
    With a Coxeter matrix $\mathcal{M}$, we will associate a set of relations that will define the \defit{Artin-Tits} monoid of type $\mathcal{M}$.
    \begin{definition}
        The Artin-Tits monoid $M=M_{\Gamma}=M_{\mathcal{M}}$ associated with a Coxeter matrix $\mathcal{M}$ (and a Coxeter graph $\Gamma$) is the monoid with the presentation $\pres{S}{R}^+$, where $R$ is the set of relations of the form:
        \[
        \underbrace{aba\ldots}_{m_{a,b} \text{ factor}} = \underbrace{bab\ldots}_{m_{a,b} \text{ factor}}
        \]
        for all $a, b \in S$ such that $a \neq b$ and $m_{a,b} \ne \infty$.
        If we add the relation $a^2=1$ for all $a\in S$, the monoid given by this presentation is a group $W$ and is called the \defit{Coxeter group} associated with $\mathcal{M}$. Then $(W,S)$ is called a Coxeter system.
        The Coxeter matrix (resp. the Coxeter diagram) is said to be of \defit{spherical type} if $W$ is finite.
    \end{definition}
    The Coxeter group $W$ has a faithful action $\rho$ on $\R^n$ with basis $\{\alpha_1, \ldots, \alpha_n\}$ where generators act as reflections through the hyperplanes orthogonal to the basis vectors with respect to the bilinear form defined by:
    \[
    (\alpha_i|\alpha_j)= \begin{cases} -\cos(\pi/m_{i,j}) & \text{if } m_{i,j} \ne \infty \\ -1 & \text{if } m_{i,j} = \infty \end{cases}.
    \]
    The action is defined by $\rho(s_i)(v)=v-2(v|\alpha_i)\alpha_i$. This action is faithful and preserves the bilinear form. We identify $W$ with its image by $\rho$, in the orthogonal group of $(\R^n, (\cdot|\cdot))$. For any $w \in W$ the subspace of fixed points of $w$ is denoted by $K(w) = \ker(w - \text{Id})$. An element $t \in W$ is called a reflection if codim$(K(t))=1$. The set of reflections is denoted by $T$. If $t\in T$ is a reflection and $w\in W$ then ${}^wt = wtw^{-1}$ is also a reflection. The set $S$ of generators is a subset of $T$. If $s,t \in T$ are two reflections, then $st$ is a rotation, and we denote by $m_{s,t}$ the order of $st$ in $W$ which is always finite in the spherical case. We will write $s \bot t$ if $s$ and $t$ commutes, i.e. when $m_{s,t} = 2$.

\section{Chambers, Coxeter elements}\label{sec:chamber_coxeter}
    There have been multiple definitions of Coxeter elements in the literature, and some confusion can arise between them. Each definition is natural in its context, but the confusion comes from the fact that we can consider a Coxeter system or the associated reflection group, and from the fact that the group is identified with its geometric representation or not. We will distinguish three types of Coxeter elements, that are defined in different contexts.\\
    \begin{definition}
        A \defit{reflection} is an element $s$ of order $2$ in $\GL(V)$ where $V$ is a real finite dimensional vector space, such that $K(s)=\ker(s - \text{Id})$ is a hyperplane. A \defit{geometric reflection group} is a subgroup $W$ of $\GL(V)$ generated by reflections. An \defit{abstract reflection group} is a pair $(W, T)$ where $W$ is a group and $T \subset W$ is a generating set, such that there exists a faithful representation $\rho : W \to \GL(V)$ such that $\rho(W)$ is a geometric reflection group and $\rho(T)$ is the set of reflections in $\rho(W)$. Unless specified otherwise, we will always assume that $W$ is finite.
    \end{definition}
    The classification of finite (abstract) reflection groups by Coxeter shows that a pair $(W,T)$ is an abstract reflection group if and only if there exists $S \subset T$ such that $(W,S)$ is a Coxeter system of spherical type, and $T$ is the closure of $S$ under conjugation. The number $n = |S|$ only depends on $W$ and is called the rank of $W \subset \GL(V)$, of $(W,T)$ and of $(W,S)$. The set $S$ is called a \defit{Coxeter generating set}. The set of all Coxeter generating sets of $(W,T)$ is denoted by $\cgs(W,T)$.
    \begin{remark}
        If $(W,T)$ is an abstract reflection group, one could have a Coxeter system $(W, S)$ where $S$ is not a subset of $T$. Then we will not call $S$ a Coxeter generating set for $(W,T)$. For example in type $A_5$, $W=S_6$ has outer automorphisms that send transpositions to products of three disjoint transpositions and provides a Coxeter system $(W, S)$ where $S$ is not a subset of $T$.   
    \end{remark}
    \begin{example}\label{ex:A3}
        In rank $3$, the chambers of the geometric reflection group $W$ can be represented in the plane in the following way : For each reflection $t \in T$, the (hyper)plane $K(t)$ intersect the unit sphere in a great circle. All these circle delimit surfaces on the unit sphere, which are exactly the intersection of the chambers with the unit sphere. Thus, each chamber is a triangle on the unit sphere. With a stereographic projection one can visualize the chambers in 2D. The Figure \ref{fig:A3_chambers} shows the chambers of the geometric reflection group $W$ of type $A_3$. 
        \begin{figure}[ht]
            \centering
            \includegraphics[width=0.5\textwidth]{"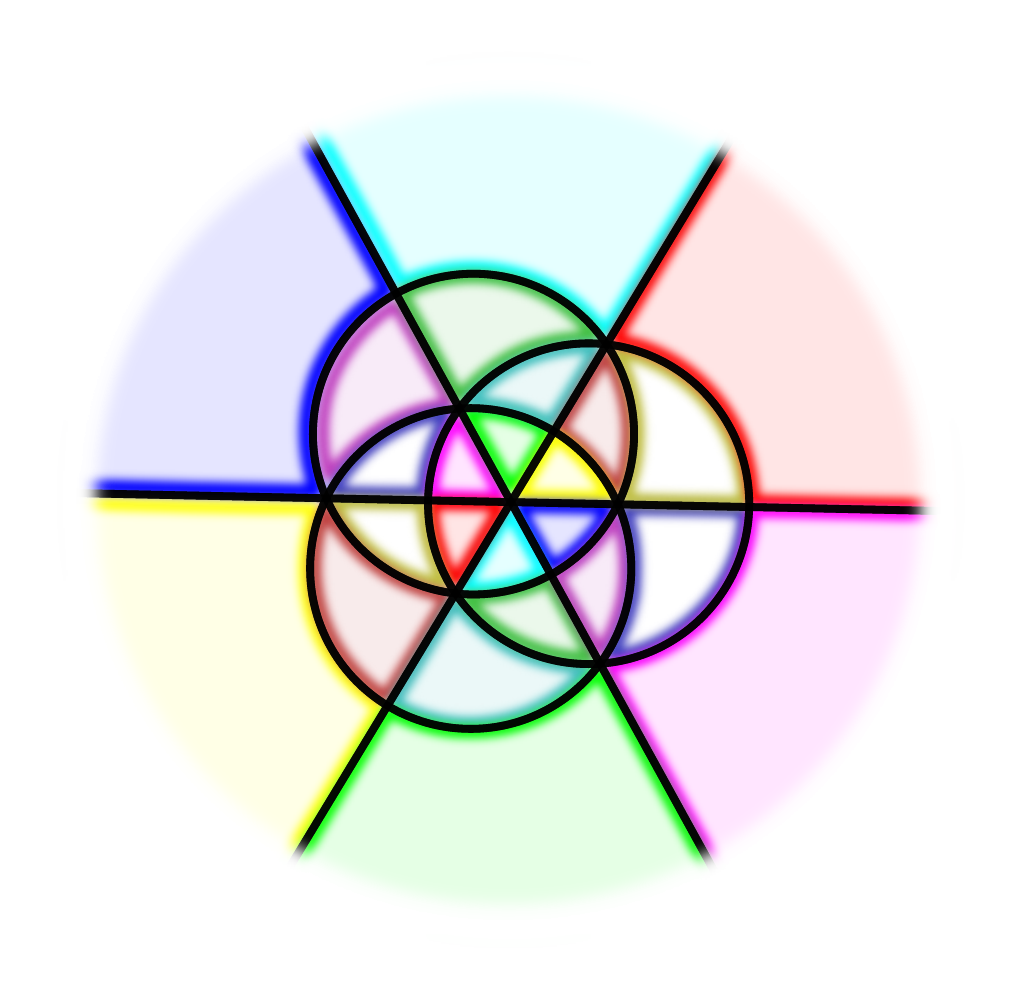"}
            \caption{Chambers of the geometric reflection group $W$ of type $A_3$. Chambers of the same color are opposite.}\label{fig:A3_chambers}
        \end{figure}
    \end{example}
    The first definition of Coxeter element is in the context of a Coxeter system.
    \begin{definition}
        Let $(W,S)$ be a Coxeter system of spherical type. A \defit{simple Coxeter element} for $S$ (or $S$-\defit{simple Coxeter element}) is an element $c \in W$ that can be written as a product of the elements of $S$ in some order, i.e. $c = s_1 s_2 \ldots s_n$ where $\{s_1, \ldots, s_n\} = S$. We will denote by $\C(W,S)$ the set of simple Coxeter elements for $S$.
    \end{definition}
    One can show that all simple Coxeter elements are conjugate \cite{bourbakiGroupesAlgebresLie2007} (chap. V, \S 6, n°1 Proposition 1), but the elements of the conjugacy class of an $S$-simple Coxeter element are not all $S$-simple Coxeter element. For example, in type $A_n$, $W$ is the symmetric group $\Sym_{n+1}$, and $(W,S)$ is a Coxeter system where $S$ is the set of transpositions $(i\ i+1)$ for $i=1, \ldots, n$. Then there are $2^{n-1}$ simple Coxeter elements, corresponding to orientations of the Coxeter diagram, but there are $n!$ elements in the conjugacy class of $c$ (the $(n+1)$-cycles). The natural extension of this definition when working with the set of all reflections is the following.
    \begin{definition}
        Let $(W,T)$ be an abstract reflection group. Then an \defit{abstract Coxeter element} of $(W,T)$ is an element $c \in W$ such that there exists a Coxeter generating set $S \subset T$ and $c$ is a simple Coxeter element $S$. We will denote by $C(W,T)$ the set of abstract Coxeter elements of $(W,T)$. Thus, $C(W,T) = \bigcup_{S} \C(W,S)$ where the union is over all Coxeter generating sets $S \subset T$.
    \end{definition}
    The set of abstract Coxeter elements is closed under any automorphism of $W$ that preserves $T$. Let $\Aut_T(W)$ denote this set of automorphisms of $W$. Then $\Inn(W) \subset \Aut_T(W)$, so $C(W,T)$ is closed under conjugation, contrary to $\C(W,S)$. However, two abstract Coxeter elements may not be conjugate. Moreover, since for every two Coxeter generating sets $S, S' \subset T$ there is a map $S \to S'$ that preserves the Coxeter matrix and extends to an automorphism of $W$ that preserves $T$, $Aut_T(W)$ acts transitively on $\cgs(W,T)$, thus on the set $C(W,T)$.
    \begin{example}
        \label{ex:H2}(cf.~\cite[Example 1.2]{reinerNonconjugateCoxeterElements2015})\\
        For instance, in $I_2(5) = \pres{\{s,t\}}{(st)^5, s^2, t^2}$ the symmetry group of the regular pentagon, then $S=\{s,t\}$ and $S'=\{s, tst\}$ are two Coxeter generating sets and the abstract Coxeter elements $c = st$ and $c^2 = stst$ are not conjugate in $W$. See Figure \ref{fig:H2}.
        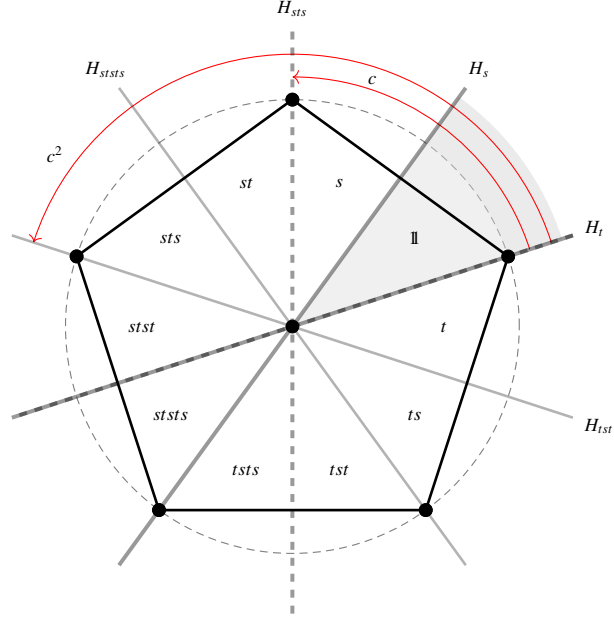
\begin{figure}[t]
            \begin{center}
            \begin{tikzpicture}[scale=1.2]
                \tikzstyle{hp}=[line width=1pt,opacity=0.3,black]
                \tikzstyle{hp2}=[line width=1.5pt,opacity=0.4,black]
                \tikzstyle{hp3}=[line width=1.5pt, opacity=0.4,black,dashed]
                \coordinate (X) at (0,0);
                \polygon{(X)}{obj}{5}{2.5}{,,,,}
                \fill[fill=black!8!white]
                    (0,0) -- ($1.25*(obj1)$) arc (18:54:3.125);
                \draw[red,->]
                    ($1.1*(obj1)$) arc (18:90:2.75);
                \node (c) at ($0.6*($(obj5)-(obj3)$)$) {\scriptsize $c$};

                \draw[red,->]
                    ($1.2*(obj1)$) arc (18:162:3);
                \node (c2) at ($0.68*($(obj4)-(obj2)$)$) {\scriptsize $c^2$};

                \polygon{(X)}{obj}{5}{2.5}{,,,,}
                \draw[hp2] ($1.3*(obj1)$) to ($-1.3*(obj1)$);
                \draw[hp3] ($1.3*(obj1)$) to ($-1.3*(obj1)$);
                \draw[hp] ($1.3*(obj2)$) to ($-1.3*(obj2)$);
                \draw[hp2] ($1.3*(obj3)$) to ($-1.3*(obj3)$);
                \draw[hp] ($1.3*(obj4)$) to ($-1.3*(obj4)$);
                \draw[hp3] ($1.3*(obj5)$) to ($-1.3*(obj5)$);
                \draw[line width=1pt,black] (obj1) to (obj2) to (obj3) to (obj4) to (obj5) to (obj1);

                \filldraw[black] (0,0) circle(2pt);

                \node (e)  at ($0.35*($(obj1)-(obj3)$)$) {\scriptsize $\one$};

                \node (s) at ($0.35*($(obj5)-(obj3)$)$) {\scriptsize $s$};
                \node (t) at ($0.35*($(obj1)-(obj4)$)$) {\scriptsize $t$};

                \node (st) at ($0.35*($(obj5)-(obj2)$)$) {\scriptsize $st$};
                \node (ts) at ($0.35*($(obj2)-(obj4)$)$) {\scriptsize $ts$};

                \node (sts) at ($0.35*($(obj4)-(obj2)$)$) {\scriptsize $sts$};
                \node (tst) at ($0.35*($(obj2)-(obj5)$)$) {\scriptsize $tst$};

                \node (stst) at ($0.35*($(obj4)-(obj1)$)$) {\scriptsize $stst$};
                \node (tsts) at ($0.35*($(obj3)-(obj5)$)$) {\scriptsize $tsts$};

                \node (ststs) at ($0.35*($(obj3)-(obj1)$)$) {\scriptsize $ststs$};

                \node (Ht) at ($1.4*(obj1)$) {\scriptsize $H_t$};
                \node (Hs) at ($-1.4*(obj3)$) {\scriptsize $H_s$};
                \node (Hsts) at ($1.4*(obj5)$) {\scriptsize $H_{sts}$};
                \node (Htst) at ($-1.42*(obj4)$) {\scriptsize $H_{tst}$};
                \node (Hststs) at ($-1.4*(obj2)$) {\scriptsize $H_{ststs}$};
            \end{tikzpicture}
            \end{center}
            \caption{The symmetry group of the regular pentagon.}
            \label{fig:H2}
        \end{figure}
    \end{example}
    The last definition is in the context of geometric reflection groups.
    \begin{definition}
        Let $W \subset \GL(V)$ be a finite geometric reflection group, and let $T$ be the set of reflections in $W$. Let $\mathcal{A} = \{K(t) | t \in T\}$ be the set of reflecting hyperplanes. The connected components of $V \setminus \cup\mathcal{A}$ are called \defit{chambers}. If $C$ is a chamber, and $t\in T$, then $H = K(T)$ is called a \defit{wall} of $C$ if $\bar{C}\cap H $ contains a non-empty open subset of $H$. If $C$ is a chamber, we denote by $S(C)$ the set of reflections whose hyperplane is a wall of $C$. Such a set $S$ is classically a Coxeter generating set, and is called a \defit{cameral generating set} (or a \defit{cam-set}). A \defit{cameral Coxeter element} or a simple Coxeter element for some cameral generating set. We will denote by $\C^c(W)$ the set of cameral Coxeter elements of $W$.
    \end{definition}
    This last definition depends on the geometric representation of $W$. For instance, in Example \ref{ex:H2}, with the geometric representation of Figure \ref{fig:H2}, $c$ is cameral but $c^2$ is not. Cameral Coxeter elements form exactly one conjugacy class in $W$.
    \begin{proposition}\label{prop:cam-set_one_orbit}
        Let $W \subset \GL(V)$ be a finite geometric reflection group. The set of cam-sets form a single orbit under the action of $W$ by conjugation on $\mathcal{P}(W)$. As a consequence, the set of cameral Coxeter elements $\C^c(W)$ is a single conjugacy class in $W$.    
    \end{proposition}
    \begin{proof}
        The group $W$ acts simply transitively on the set of chambers (see for example \cite{bourbakiGroupesAlgebresLie2007},chap. V, \S 3, n°3 Théorème 2). Moreover, the action of $W$ on hyperplanes corresponds to the conjugation action on reflections, i.e. for all $w \in W$ and $s \in T$, $K({}^ws) = w(K(s))$. Thus, if $C$ is a chamber and $w \in W$, then $S(w.C) = {}^w S(C)$. Then, conjugation is transitive on the set of cam-sets. Now, if $c, c'$ are two cameral Coxeter elements, there exist two cam-set $S, S'$ such that $c$ (resp. $c'$) is a simple Coxeter element for $S$ (resp. $S'$). There exists $w \in W$ such that $S' = w.S$. Then ${}^wc$ is a simple Coxeter element for $S'$ so ${}^wc$ is conjugate to $c'$. Thus, $c$ and $c'$ are conjugate. Now, if $c$ is a cameral Coxeter element and $w \in W$, then ${}^wc$ is a simple Coxeter element for ${}^w S(C)$ where $C$ is a chamber such that $c$ is a simple Coxeter element for $S(C)$. Thus, ${}^wc$ is also a cameral Coxeter element. So $\C^c(W)$ is a whole conjugacy class.
    \end{proof}
    The order of any Coxeter element (simple, abstract or cameral) is the same number $h$ called the \defit{Coxeter number} of $W$.
    \begin{remark}\label{ref:chamber_to_cam-set}
        Note that the correspondence between chambers and cam-sets is not one-to-one. Indeed, two opposite chambers share the same cam-set. In general, the number of chambers associated to a cam-set is $2^k$ where $k$ is the number of connected components of the Coxeter diagram associated to the cam-set.
    \end{remark}
    The definition given by Bessis in \cite{bessisDualBraidMonoid2003a} of Coxeter elements is different from the three above, it uses the notion of chromatic pair. But it is equivalent to our definition of abstract Coxeter element. Let's recall this definition and call it a \defit{chromatic Coxeter element}.
    \begin{definition}
        Let $(W, T)$ be an abstract finite reflection group, and let $L, R$ be disjoint subsets of $T$. We say that $(L,R)$ is a \defit{chromatic pair} if and only if:
        \begin{enumerate}[(i)]
            \item $S = L \sqcup R$ is a Coxeter generating set, i.e. $(W, S)$ is a Coxeter system.
            \item For all $s,t \in L$ (resp. $R$), $m_{s,t}=2$.
        \end{enumerate}
        Then, the \defit{chromatic Coxeter element} associated with the chromatic pair $(L,R)$ is the element $c_{L,R} := s_L s_R$ where $s_L := \prod_{s\in L}s$ and $s_R := \prod_{s\in R}s$.
    \end{definition}
    The definition of chromatic pairs is stable by any automorphism of $W$ that preserves $T$. Indeed, if $\phi \in \Aut_T(W)$ and $L,R$ is a chromatic pair, then $\phi(S)=\phi(L)\sqcup \phi(R)$ is a Coxeter generating set. Moreover, as explained above, $\Aut_T(W)$ acts transitively on $C(W,T)$, which allows one to prove the following proposition.
    \begin{proposition}
        Let $(W,T)$ be an abstract finite reflection group. Then the set of chromatic Coxeter elements is exactly the set of abstract Coxeter elements $C(W,T)$.
    \end{proposition}
    \begin{proof}
        The fact that chromatic Coxeter elements are abstract Coxeter elements is immediate from the definition. Conversely, let $c \in C(W,T)$ be an abstract Coxeter element. Let $L_0, R_0$ be a chromatic pair, and let $c_0 = s_{L_0} s_{R_0}$ be the associated chromatic Coxeter element. Then, there exists $\phi \in \Aut_T(W)$ such that $\phi(c_0) = c$. Then, $\phi(L_0), \phi(R_0)$ is a chromatic pair and $c = \phi(c_0) = s_{\phi(L_0)} s_{\phi(R_0)}$ is the associated chromatic Coxeter element.
    \end{proof}
    \subsection{Cameral decomposition}
        Let $(W,T)$ be an abstract reflection group of rank $n$.
        \begin{definition}
            Let $c$ be an abstract Coxeter element. An $n$-tuple $(t_1, \ldots, t_n)$ such that $t_1 t_2 \ldots t_n = c$ is called a \defit{reduced T-decomposition} of $c$. Moreover, if $\{t_1, \ldots, t_n\}$ is a Coxeter generating set, then we say that $(t_1, \ldots, t_n)$ is a \defit{Coxeter decomposition} of $c$. If $W$ is a geometric reflection group, and if $\{t_1, \ldots, t_n\}$ is a cameral generating set, then we say that $(t_1, \ldots, t_n)$ is a \defit{cameral decomposition} of $c$. We will denote by $\Red_T(c)$ the set of reduced T-decompositions of $c$, and by $\Red_T^c(c)$ the set of cameral decompositions of $c$.
        \end{definition}
        We will see that any abstract Coxeter element has non-Coxeter T-decompositions. On the other hand, by definition, a Coxeter decomposition of a non-cameral abstract Coxeter element is never cameral. However, one can ask if a cameral Coxeter element may have Coxeter decompositions that are not cameral. This question is intimately linked to the study of automorphisms of $W$ that preserve $T$. 
    \subsection{Automorphisms of reflection groups}
        The study of automorphisms of reflection groups done in \cite{marinAutomorphismsComplexReflection2007} and in \cite{reinerNonconjugateCoxeterElements2015} gives the answer to our question.\par
        Let's recall the results from \cite{marinAutomorphismsComplexReflection2007}. Let $W \subset \GL(V)$ be a finite geometric reflection group, and $T$ the set of its reflections. Denote by $K_W$ the subfield of $\R$ generated by the traces of the elements of $W$. Denote by $\Gamma = \Gal(K_W/\Q)$ its Galois group over $\Q$, and let $\chi$ be the character of $W$ (seen as a representation, i.e. $\chi : w \in W \mapsto \Tr(w) \in K_W$). Then an automorphism $a \in A = \Aut_T(W)$ is called a \defit{Galois automorphism attached to} $\gamma \in \Gamma$ if and only if $\chi \circ a = \gamma \circ \chi$. The first main result of \cite{marinAutomorphismsComplexReflection2007} (Theorem 1.2) is that any automorphism in $\Aut_T(W)$ is a Galois automorphism. Thus, there is a group morphism $A \to \Gamma$ that sends an automorphism attached to $\gamma\in \Gamma$ to $\gamma$. We denote by $N$ the subgroup of $A$ consisting of Galois automorphisms attached to the identity, i.e. the kernel of this morphism. It contains exactly the conjugacy by elements of the normalizer of $W$ in $\GL(V)$, including the inner automorphisms. We have the following inclusions: $\Inn(W) \subset N \subset A \subset \Aut(W)$. Taking the quotient by $\Inn(W) = W/ZW$, we denote by $\bar{A}, \bar{N}$ the images of $A$ and $N$ in $\Out(G)$. The elements of $\bar{N}$ are called diagram automorphisms, and $\bar{N}$ is normal in $\bar{A}$ because $N$ was normal in $A$. Then Theorem 1.3 in  \cite{marinAutomorphismsComplexReflection2007} state that the map $A \to \Gamma$ (or equivalently, the quotient $\bar{A} \to \Gamma$) is surjective, and has a section, i.e. there is a decomposition $\bar{A} = \bar{N} \rtimes \Gamma$.\par
        Before stating the result of \cite{reinerNonconjugateCoxeterElements2015}, we need to show that $\bar{N}$ acts trivially on $\C(W,S)/\Inn(W)$. We have seen that $A=\Aut_T(W)$ acts transitively on $\C(W,T)$, thus $\bar{A}$ acts transitively on the conjugacy classes of abstract Coxeter elements $\C(W,T)/\Inn(W)$. We want to show that $\bar{N}$ is in the kernel of this action. In other words we need the following lemma.
        \begin{lemma}\label{lem:N_acts_trivially}
            Let $c,c' \in W$ be two abstract Coxeter elements that are conjugated in $\GL(V)$. Then they are conjugated in $W$.
        \end{lemma}
        \begin{proof}
            It has been pointed out in \cite{reinerNonconjugateCoxeterElements2015} that this can be proven using Springer's theory, but we will present a direct proof that is useful for the following.\par
            The actions of $A$ on $\cgs(W,T)$ and on $\C(W,T)$ induce an action of $\bar{A}$ on the conjugacy classes of Coxeter generating sets: $\cgs(W,T)/\Inn(W)$, and on conjugacy classes of abstract Coxeter elements: $\C(W,T)/\Inn(W)$. Let $f$ be the map $\cgs(W,T)/\Inn(W) \to \C(W,T)/\Inn(W)$ that sends a conjugacy class of Coxeter generating set to the conjugacy class of its simple Coxeter elements. It is well-defined since all simple Coxeter elements with respect to a fixed cam-set are conjugate. Moreover, $f$ is $\bar{A}$-equivariant. 
\[\begin{tikzcd}[column sep=tiny,row sep=small]
	&&& {} &&&& \\
	& A &&&&& {\bar{A}} \\
	{\cgs(W,T)} && {\C(W,T)} &&& {\cgs(W,T)/\Inn(W)} && {\C(W,T)/\Inn(W)}
	\arrow["{\curvearrowright}"{marking, allow upside down}, draw=none, from=2-2, to=3-3]
	\arrow["{\curvearrowright}"{marking, allow upside down}, draw=none, from=2-7, to=3-8]
	\arrow["{\curvearrowleft}"{marking, allow upside down}, draw=none, from=3-1, to=2-2]
	\arrow["{\bullet/\Inn(W)}", shift left=5, between={0.1}{0.9}, from=3-3, to=3-6]
	\arrow["{\curvearrowleft}"{marking, allow upside down}, draw=none, from=3-6, to=2-7]
	\arrow["f"', from=3-6, to=3-8]
\end{tikzcd}\]
            We claim that $\bar{N}$ acts trivially on the $\cgs(W,T)/\Inn(W)$. Indeed, if $n \in N$ is the conjugacy by some element $g \in \GL(V)$, then $g$ permutes the walls $(K(t))_{t\in T}$ since conjugacy by $g$ preserves $T$. As a homeomorphism, $g$ permutes the connected components of $V\setminus \cup K(t)$, so it sends chambers to chambers, thus $n$ sends cameral generating sets to cameral generating sets. Moreover, $N$ is normal in $A$, thus $\bar{N}$ is normal in $\bar{A}$, and $\bar{A}$ acts transitively on $\cgs(W,T)/\Inn(W)$, thus if $\bar{N}$ fixes one element of $\cgs(W,T)/\Inn(W)$, it fixes them all. Since it fixes the conjugacy class of cam-sets, $\bar{N}$ acts trivially on $\cgs(W,T)/\Inn(W)$. Finally, since $f$ is surjective $\bar{N}$ acts trivially on $\C(W,T)/\Inn(W)$.
        \end{proof}
        We know thanks to Lemma \ref{lem:N_acts_trivially} that the actions of $\bar{A}$ on $\C(W,T)/\Inn(W)$ and on $\cgs(W,T)/\Inn(W)$ induce actions of the quotient $\bar{A}/\bar{N}$ on these two sets. Then Theorem 5.2, (i) in \cite{reinerNonconjugateCoxeterElements2015} state that the action of $\Gamma = \bar{A}/\bar{N}$ on $\C(W,T)/\Inn(W)$ is free. We can now answer whether a cameral Coxeter element may have Coxeter decompositions that are not cameral.
        \begin{proposition}\label{prop:cam_cox_cam_dec}
            Let $c$ be a cameral Coxeter element. Then, all Coxeter decompositions of $c$ are cameral decompositions.
        \end{proposition}
        \begin{proof}
            The statement reduces to showing that the map $f$ introduced in the proof of Lemma \ref{lem:N_acts_trivially} is injective. Recall that $f$ is the map from $\cgs(W,T)/\Inn(W)$ to $\C(W,T)/\Inn(W)$ that sends a conjugacy class of Coxeter generating set to the conjugacy class of its simple Coxeter elements. Let's $\bar{S},\bar{S'}$ be two conjugacy classes of Coxeter generating sets which generate the same conjugacy class of Coxeter element, i.e. $\bar{S}, \bar{S'} \in \cgs(W,T)/\Inn(W)$ such that $f(\bar{S}) = f(\bar{S'})$. Since $\bar{A}/\bar{N}$ acts transitively on $\cgs(W,T)$, there exists $g \in \Gamma$ such that $\bar{S'} = g.\bar{S}$. Then $g.f(\bar{S}) = f(g.\bar{S}) = f(\bar{S'}) = f(\bar{S})$, thus $g = 1$ since the action of $\bar{A}/\bar{N}$ on $\C(W,T)/\Inn(W)$ is free. Then $\bar{S} = \bar{S'}$. Thus, $f$ is injective, (and thus bijective). As a consequence, the conjugacy class of cameral Coxeter elements, $\C^c(W)$, has only one preimage by $f$, which is the conjugacy class of cameral generating sets. Thus, if $c$ is a cameral Coxeter element, all Coxeter generating sets $S$ such that $c$ is a simple Coxeter element for $S$ are cameral generating sets.
        \end{proof}
        \begin{remark}
            The results of \cite{reinerNonconjugateCoxeterElements2015} imply that, for Weyl groups, which are exactly the finite reflection groups for which the field $K_W$ is $\Q$, there is only one conjugacy class of Coxeter generating sets, and of Coxeter elements, the cameral ones.
        \end{remark}

\section{The Hurwitz action on T-Decompositions}\label{sec:hurwitz_on_dec}
    The Hurwitz action is an action of the braid group on $n$ strands $B_n$ on the set of ordered $n$-tuples of elements in a group that preserves the product in order.\\
    If $(W,T)$ is a finite reflection group of rank $n$, for all $s, t$ in $T$, we will denote by ${}^ts=tst$ the conjugate of $s$ by $t$. The braid group $B_n$ acts on the $n$-tuples of distinct elements of $T$ by the Hurwitz action defined as follows:
    \[
    \braid_i(s_1,\ldots,s_n) = (s_1,\ldots,s_{i-1},s_{i+1},{}^{s_{i+1}}s_i, s_{i+2},\ldots,s_n) \text{ for all } 1 \leq i < n.
    \]
    The Hurwitz action restricted to the set of reduced $T$-decompositions of an abstract Coxeter element is transitive. The first reference of this fact is a letter from P. Deligne to E. Looijenga \cite{deligneDeligneLooijenga1974}, and the first published proof is due to Bessis in \cite{bessisDualBraidMonoid2003a}. However, this action does not preserve the underlying set $S$ of the decomposition, and it does not even preserve the property of being a Coxeter generating set:
    \begin{example}\label{exa:hurwitz_preserv_no_cox}
        Let $W$ be the Coxeter group of type $A_3$, with $S = \{u, s, t\}$, and let $c = ust$. Then $(u, s, t)$ is a cameral decomposition of $c$. Applying the Hurwitz action by $\braid_1$ gives: $\braid_1.(u, s, t) = (s, {}^su, t) = (s, usu, t)$. Then, $s \times usu = us$ has order $3$, $usu \times t = {}^u(st)$ has order $3$ and $st$ has order $3$. Thus, $\{s, usu, t\}$ is not a cam-set, or it would be conjugate to $S$ and would contain two commuting reflections.
    \end{example}
    But we can restrict this action to preserve the fact of being a cam-set. We just allow two rules: (a) swap reflections that commute, (b) move the last reflection from the far right to the far left (this has the effect of conjugating the set $S$ of reflections in the $n$-tuple).\\
    Formally, we define the relations $\toa$ and $\tob$ on the set of ordered cam-sets by the following rules:\\
    (a) If $s_i \bot s_{i+1}$, then $(s_1,\ldots,s_n) \toa (s_1,\ldots,s_{i-1},s_{i+1},s_i,s_{i+2},\ldots,s_n)$\\
    (b) $(s_1,\ldots,s_n) \tob (s_n, s_ns_1s_n, s_ns_2s_n,\ldots,s_ns_{n-1}s_n)$\\
    Rule $(a)$ is involutory, i.e. relation $\toa$ is symmetric, and has no effect on the underlying cam-set, while rule $(b)$ conjugates the underlying cam-set by $s_n$. Applying rule $(b)$ $n$ times has the effect of conjugating the ordered cam-set by the product of its elements, which is a Coxeter element. Applying rule $(b)$ $nh$ times thus has the effect of conjugating the cam-set by $c^h = 1$, thus rule $(b)$ is "cyclic", in particular,  its transitive closure is symmetric. Thus, the reflexive transitive closure of $\toa$ (resp. $\tob$) is an equivalence relation, that we will denote $\toas$ (resp. $\tobs$). We will denote by $\eqab$ the equivalence relation generated by both $\toa$ and $\tob$.\\
    Our main result is that all cameral decompositions of a cameral Coxeter element are $\eqab$-equivalent. We will introduce the formalism of partial actions of groups to prove this statement.

\section{Partial actions of groups}\label{sec:partial_actions_of_groups}
    \subsection{Generalities}
        We define the notion of partial group action as in \cite{kellendonkPartialActionsGroups2004}.
        \begin{definition}
            A \defit{partial action} of a group $G$ on a set $X$ is a partial map $\rho: G \times X \to X$, denoted by $\rho(g,x) = g.x$ whenever it is defined, such that:\\
            (PA0) For all $x \in X$, $1.x$ is defined and $1.x = x$.\\
            (PA1) For all $x \in X$, if $g.x$ is defined, then $g^{-1}.(g.x)$ is defined and $g^{-1}.(g.x) = x$.\\
            (PA2) For all $x \in X$, if $h.x$ and $g.(h.x)$ are defined, then $(gh).x$ is defined and $(gh).x = g.(h.x)$.\\
            The axiom (PA0) is sometimes omitted, and a partial action satisfying (PA0) is called \defit{unital}. Here, all partial actions are unital. A partial action is total if $\rho$ is a total map. Then $\rho$ is a group action of $G$ on $X$. A partial action $\rho$ is called \defit{transitive} if for all $x,y \in X$, there exists $g \in G$ such that $g.x$ is defined and equal to $y$. If $x \in X$, the \defit{orbit} of $x$ is the set $\orbit{x} = \{y \in X| \exists g \in G , \alpha(g,x)=y\}.$ The domain of definition of $\rho$ is the set $\domain{\rho} = \{(g,x) \in G \times X | \rho(g,x) \text{ is defined}\}.$
        \end{definition}
        We fix a group $G$ acting partially on a set $X$ by a partial action $\alpha$. We can define the notion of sub-partial action and get classic sub-structures properties, restriction, intersection, sub-partial action generated by a subset, as follows.
        \begin{definition}
            Let $\beta$ be a partial action of $G$ on $X$. We say that $\beta$ is a \defit{sub-partial action} of $\alpha$ if for all $g \in G$ and $x \in X$, if $\beta(g,x)$ is defined, then $\alpha(g,x)$ is defined and $\beta(g,x) = \alpha(g,x)$. We denote $\beta \subset \alpha$.
        \end{definition}
        The restriction of a partial action on a subset of $X$ is defined as follows:
        \begin{definition}
            Let $Y \subset X$ be a subset. We define the \defit{restriction} of $\alpha$ to $Y$, denoted by $\alpha|_Y$, as the partial map defined by: for all $g \in G$ and $y \in Y$, $\alpha|_Y(g,y)$ is defined if and only if $\alpha(g,y)$ is defined and $\alpha(g,y) \in Y$, and in this case, $\alpha|_Y(g,y) = \alpha(g,y)$. Then $\alpha|_Y \subset \alpha$.
        \end{definition}
        \begin{proof}
            We have to check the three axioms of partial actions.\\
            (PA0) For all $y \in Y$, $\alpha|_Y(1,y)$ is defined and equal to $y \in Y$.\\
            (PA1) Let $y \in Y$ and $g \in G$ such that $\alpha|_Y(g,y)$ is defined. Then, $\alpha(g,y)$ is defined and in $Y$, so by (PA1), $\alpha(g^{-1}, \alpha(g,y))$ is defined and equal to $y \in Y$.\\
            (PA2) Let $y \in Y$ and $g,h \in G$ such that $\alpha|_Y(h,y)$ and $\alpha|_Y(g, \alpha|_Y(h,y))$ are defined. Then, $\alpha(h,y)$ is defined and in $Y$, and since $\alpha|_Y(h,y) = \alpha(h,y)$, we have that $\alpha(g, \alpha(h,y))$ is defined and in $Y$. Thus, by (PA2), $\alpha(gh, y)$ is defined and equal to $\alpha(g, \alpha(h,y))$. Since $\alpha(g, \alpha(h,y)) \in Y$, we have that $\alpha|_Y(gh, y)$ is defined and equal to $\alpha|_Y(g, \alpha|_Y(h,y)).$
        \end{proof}
        Intersection of sub-partial actions which corresponds to the intersection of their domains of definition.
        \begin{definition}
            Let $B$ be a set of sub-partial actions of $\alpha$, and let $\gamma$ be the partial action defined by: for all $g \in G$ and $x \in X$, $\gamma(g,x)$ is defined if and only if for all $\beta \in B$, $\beta(g,x)$ is defined, and in this case, $\gamma(g,x) = \alpha(g,x)$. Then, $\gamma \subset \alpha$, denoted by $\gamma = \bigcap_{\beta \in B} \beta$, called the \defit{intersection} of the partial actions in $B$.
        \end{definition}
        \begin{proof}
            (PA0) For all $x \in X$, for all $\beta \in B$, $\beta(1,x)$ is defined and equal to $x$, thus $\gamma(1,x)$ is defined and equal to $x$.\\
            (PA1) Let $x \in X$ and $g \in G$ such that $\gamma(g,x)$ is defined. Then, for all $\beta \in B$, $\beta(g,x)$ is defined, so by (PA1) for all $\beta \in B$, $\beta(g^{-1}, \beta(g,x))$ is defined and equal to $x$. Thus, $\gamma(g^{-1}, \gamma(g,x))$ is defined and equal to $x$.\\
            (PA2) Let $x \in X$ and $g,h \in G$ such that $\gamma(h,x)$ and $\gamma(g, \gamma(h,x))$ are defined. Then, for all $\beta \in B$, $\beta(h,x)$ is defined, and since $\gamma(h,x) = \beta(h,x)$, we have that for all $\beta \in B$, $\beta(g, \beta(h,x))$ is defined. Thus, by (PA2), for all $\beta \in B$, $\beta(gh, x)$ is defined and equal to $\beta(g, \beta(h,x))$. Thus, $\gamma(gh, x)$ is defined and equal to $\gamma(g, \gamma(h,x))$.
        \end{proof}
        Finally, we can define the sub-partial action generated by a set of elements of $G \times X$.
        \begin{definition}
            Let $R \subset G \times X$ be a set of pairs $(g,x) \in \domain(\alpha)$. Let $B$ be the set of sub-partial actions $\beta$ of $\alpha$ such that  $\beta(g,x)$ is defined for all $(g,x) \in R$, and let $\gamma = \bigcap_{\beta \in B} \beta$. Then by the above lemma, $\gamma$ is a sub-partial action of $\alpha$, called the \defit{partial action generated by} $R$.
        \end{definition}
        A useful description of the sub-partial action generated by a set is given by the following proposition:
        \begin{proposition}\label{prop:generated_partial_action}
            Let $R \subset G \times X$ be a set of pairs $(g,x) \in \domain(\alpha)$. Let $\gamma$ be the sub-partial action of $\alpha$ generated by $R$. Let $g \in G$ and $x \in X$. Then, $\gamma(g,x)$ is defined if and only if there exist two finite sequences $g_1, g_2, \ldots, g_k$ in $G$ and $x_1, x_2, \ldots, x_k, x_{k+1}$ in $X$ such that $x_1 = x$, $g = g_k\ldots g_2g_1$, and for all $1\le i \le k$ $x_{i+1} = \alpha(g_i,x_i)$ and $(g_i, x_i) \in R$. 
        \end{proposition}
        \begin{proof}
            Assume first that such sequences exist. Then by (PA2) repeatedly, we have that $\gamma(g, x)$ is defined.\\
            For the converse, let $\beta$ be the sub-map of $\alpha$ defined on the subset of $G \times X$ consisting of all pairs $(g,x)$ such that there exists such sequences. We need to prove that $\domain(\gamma) \subset \domain(\beta)$. But since $R \subset \domain(\beta)$, by definition of $\gamma$, it is enough to prove that $\beta$ is a sub-partial action of $\alpha$.\\
            (PA0) For all $x \in X$, the sequence of length $k = 0$ shows that $\beta(1,x)$ is defined and equal to $x$.\\
            (PA1) Let $x \in X$ and $g \in G$ such that $y = \beta(g,x)$ is defined. Then, there exists a sequence $g_1, \ldots, g_k$ in $G$ and $x_1, \ldots, x_{k+1}$ in $X$ such that $x_1 = x$, $g = g_k\ldots g_2 g_1$, and for all $1\le i \le k$, $x_{i+1} = \alpha(g_i, x_i)$ and $(g_i, x_i) \in R$. Define now the sequence $h_1, \ldots, h_k$ in $G$ and $y_1, \ldots, y_{k+1}$ in $X$ by $h_i = g_{k-i+1}^{-1}$ and $y_i = x_{k-i+2}$ for all $1 \le i \le k+1$. Then we have $g^{-1} = h_k \ldots h_2 h_1$, and for all $1 \le i \le k$, $y_{i+1} = \alpha(h_i, y_i)$ and $(h_i, y_i) \in R$. Thus, $\beta(g^{-1}, y)$ is defined and equal to $x$.\\
            (PA2) Let $x \in X$ and $g,h \in G$ such that $y = \beta(h,x)$ and $z = \beta(g,y)$ are defined. Then, there exists a sequence $h_1, \ldots, h_k$ in $G$ and $x_1, \ldots, x_{k+1}$ in $X$ such that $x_1 = x$, $h = h_k \ldots h_2 h_1$, and for all $1\le i \le k$, $x_{i+1} = \alpha(h_i, x_i)$ and $(h_i, x_i) \in R$. Similarly, there exists a sequence $g_1, \ldots, g_l$ in $G$ and $y_1, \ldots, y_{l+1}$ in $X$ such that $y_1 = y$, $g = g_l \ldots g_2 g_1$, and for all $1\le i \le l$, $y_{i+1} = \alpha(g_i, y_i)$ and $(g_i, y_i) \in R$. Define now the sequence $f_1, \ldots, f_{k+l}$ in $G$ and $z_1, \ldots, z_{k+l+1}$ in $X$ by: for all $1 \le i \le k$, $f_i = h_i$ and $z_i = x_i$, and for all $1 \le j \le l$, $f_{k+j} = g_j$ and $z_{k+j+1} = y_{j+1}$. Then we have $gh = f_{k+l} \ldots f_2 f_1$, and for all $1 \le i \le k+l$, $z_{i+1} = \alpha(f_i, z_i)$ and $(f_i, z_i) \in R$. Thus, $\beta(gh, x)$ is defined and equal to $z$.
        \end{proof}

    \subsection{Permutation actions on trees}\label{subsec:partial_actions_and_trees}
        The main combinatorial lemma that we will use is the following. It states that one can order the vertices of a forest in any way by using only transpositions of consecutive vertices that are not connected in the graph, together with the cycle that takes the last vertex to the first position.
        \begin{lemma}\label{lem:partial_action_tree}
            Let $n\ge 1$ and let $\Gamma$ be a simple graph with $n$ vertices, and let $\orderings(\Gamma)$ the set of all bijection from $\{1, \ldots, n\}$ to $V(\Gamma)$. Let $G = \Sym_n$ be the symmetric group on $n$ elements. Denote by $\alpha$ the group action of $G$ on $\orderings(\Gamma)$ defined by: for all $\sigma \in \Sym_n$ and $f \in \orderings(\Gamma)$, $\sigma.f = \alpha(\sigma, f) = f \circ \sigma^{-1}$.\\
            Let $R_{\Gamma} = \{(f,(i\ i+1)) | (f(i), f(i+1)) \notin E(\Gamma)\}$. Then, let $R_C = \{C\}\times \orderings(\Gamma)$ where $C$ is the $n$-cycle $(1\ 2\ \ldots \ n)$. Let $\gamma$ be the sub-partial action of $\alpha$ generated by $R_{\Gamma} \cup R_C$.\\
            If $\Gamma$ is a forest, then $\gamma = \alpha$, i.e. $\gamma$ is the smallest sub-partial action of $\alpha$ whose domain contains $R_{\Gamma} \cup R_C$.
        \end{lemma}
        \begin{proof}
            Let $\gamma \subset \alpha$ such that $R_{\Gamma}\cup R_C \subset \domain(\gamma)$\\
            Notice first that for all $f$ in $\orderings(\Gamma)$, using (PA2) repeatedly with the fact that  $R_C \subset \domain(\gamma)$, we have that for all $k\ge 0$, $(C^k, f) \in \domain(\gamma)$, and using (PA1), gives the same for negative $k$.\\
            The proof is by induction on $|E(\Gamma)|$. Assume first that $|E(\Gamma)| = 0$, i.e. $E(\Gamma) = \emptyset$. Then for all $f \in \orderings(\Gamma)$ and all $i \in \{1, \ldots, n-1\}$, $(f(i), f(i+1)) \notin E(\Gamma)$, so $((i\ i+1), f) \in R_{\Gamma} \subset \domain(\gamma)$. Let $(\sigma, f) \in \orderings(\Gamma) \times \Sym_n$. We can write $\sigma$ as a product of transpositions $(i\ i+1)$: $\sigma = (i_1\ i_1+1)(i_2\ i_2+1) \cdots (i_k\ i_k+1)$. By applying (PA2) repeatedly, we get that $(\sigma, f) \in \domain(\gamma)$, so $\gamma = \alpha$.\\
            Assume now that $|E(\Gamma)| > 0$. Since $\Gamma$ is a forest with at least an edge, it has a leaf, i.e. a vertex of degree $1$. Let $v$ be a leaf of $\Gamma$, and let $u$ be the unique vertex such that $(u,v) \in E(\Gamma)$. Let $R(u,v) = \{f, (i\ i+1) | \{f(i), f(i+1)\} = \{u,v\}\}$. We will show that $R(u,v) \subset \domain(\gamma)$. Let $((i\ i+1), f) \in R(u,v)$. Assume first that $f(i) = u$ and $f(i+1) = v$ is the leaf. Then by the preliminary remark, $(C^{n-i}, f) \in \domain(\gamma)$. Let $f' = \alpha(C^{n-i}, f) = f \circ C^{i-n}$. We have $f'(n)=f(C^{i-n}(n))=f(i)=u$ and $f'(1) = f(C^{i-n}(1)) = f(i+1) = v$. Since $v$ is a leaf and $u$ is the unique vertex adjacent to $v$, we have $(u, f'(j)) \notin E(\Gamma)$ for all $2 \le j \le n-1$. Let $f'_1=f'$, and define successively, for $1 \le j \le n-2$: $f'_{j+1} = \alpha((j\ j+1), f'_j(j))$. By construction, $f'_j(j)= v$ is the leaf and $f'_j(n)=u$ so $(f'_j(j), f'_j(j+1)) \notin E(\Gamma)$, thus $((j\ j+1), f'_j) \in R_{\Gamma} \subset \domain(\gamma)$. Then, by applying (PA2) repeatedly, we have $(C', f') \in \domain(\gamma)$, where $C' = (n-2\ n-1)\ldots (1\ 2) = (n-1\ n-2\ \ldots\ 1) = (n-1\ n)C^{-1}$. Then by (PA2) again, we have $(C'\circ C^{n-i}, f) \in \domain(\gamma)$. Denote $f'' = \alpha(C' \circ C^{n-i}, f)$, and use the preliminary remark again to get that $(C^{-i+1}, f'') \in \domain(\gamma)$. Finally, by (PA2) again, we have $(C^{-i+1} \circ C' \circ C^{n-i}, f)$ in $\domain(\gamma)$. But we have $C^{-i+1} \circ C' \circ C^{n-i} = C^{-i+1} \circ (n-1\ n) \circ C^{-i+1} = (i\ i+1)$. Thus, we have shown that $((i\ i+1), f) \in \domain(\gamma)$ in the case where $f(i) = u$ and $f(i+1) = v$. The case where $f(i) = v$ and $f(i+1) = u$ reduces to the previous one by using (PA1). Thus, we have shown that $R(u,v) \subset \domain(\gamma)$. Now, let $\Gamma' = (V(\Gamma), E(\Gamma) \setminus \{(u,v)\})$ be the graph obtained from $\Gamma$ by removing the edge $(u,v)$. Since $\Gamma'$ is a forest, and $R(\bar{\Gamma'}) = R_{\Gamma} \cup R(u,v) \subset \domain(\gamma)$, then by induction hypothesis, we have $\gamma = \alpha$.
        \end{proof}
        \begin{remark}
            The previous lemma is not true if $\Gamma$ is not a forest, i.e. has a cycle. Indeed, if a cycle of the graph appears in order, then transpositions will never break it and the cycle action will only rotate it, so it will always appear up to rotation with the same orientation.
        \end{remark}

\section{Results in type $A$}
    In type $A_n$, the Coxeter group is the symmetric group $W = \Sym_{n+1}$, the set of reflections $T$ is the set of transpositions, and the Coxeter elements are the $(n+1)$-cycles. We can view the transposition $(i\ j)$ as the edge $\{i, j\}$ in the complete graph where the vertices are the integers from $1$ to $n+1$. A set of transpositions thus corresponds to a graph subgraph of the complete graph on the vertex set $[1, n+1]$.
    \begin{definition}
        A set $S$ of $n$ transpositions is called a (Hamiltonian) \defit{path}, if all vertices have degree $2$ except two vertices of degree $1$, called \defit{endpoints}. A path is characterized by the sequence of vertices from one endpoint to the other one up to reversion. The path $\{(a\ x_1), (x_1\ x_2), \ldots, (x_{n-1}\ b)\}$ will be denoted $S = a \longline x_1 \longline x_2 \longline \cdots \longline x_{n-1} \longline b$.
    \end{definition}
    \begin{proposition}\label{prop:cam-set_conjugation}
        Let $S= x_0 \longline x_1 \longline \cdots \longline x_n$. Let $\sigma \in W$ be a permutation and $S' = {}^{\sigma}S = \{{}^\sigma s | s \in S\}$ be the cam-set obtained by conjugating $S$ by $\sigma$. Then $S'=\sigma(x_0) \longline \cdots \longline \sigma(x_n)$.
    \end{proposition}
    \begin{proof}
        Let $t_j = (x_{j-1}\ x_j)$ for all $1 \le j \le n$. Then, ${}^\sigma t_j = (\sigma(x_{j-1})\ \sigma(x_j))$, so $S'= \sigma(x_0) \longline \cdots \longline \sigma(x_n)$.
    \end{proof}
    In type A, a cam-set always has two particular orderings, corresponding to the two possible orientations of the path, that should not be confused with an ordering of the cam-set that define a product and thus a Coxeter element.
    \begin{definition}\label{defi:bi_colored_cam-set}
        Assume $n > 0$. A \defit{bi-colored cam-set} is a cam-set $S$ with a partition $\{A, B\}$ of the vertices, i.e. $[1, n+1] = A\sqcup B$, such that each endpoint of the path is in a different part.\\
        Let $(S, \{A, B\})$ be a bi-colored cam-set. Write $S = x_0 \longline \cdots \longline x_n$ with $a=x_0$ the endpoint in $A$.
        There is a natural ordering on $A$ (resp. on $B$) that starts with the endpoint $a$ (resp. $b$) and reads the vertices in the order they appear in the path. Precisely, we denote by $a=a_1, a_2,  \ldots a_k$ the subsequence of $(x_i)_{0\le i \le n}$ of vertices in $A$. Symmetrically, let $b=b_1, b_2, \ldots, b_{n+1-k}$ be the reversed subsequence of $(x_i)_{0\le i \le n}$ of vertices in $B$, i.e. when we traverse the path from $b$ to $a$. Then we define the Coxeter element associated with the bi-colored cam-set by $c(S, \{A, B\}) = (a_1\ a_2\ \ldots\ a_k\ b_1\ b_2\ \ldots\ b_{n+1-k})$.
    \end{definition}
    \begin{remark}
        The classic way of choosing a Coxeter element after choosing a chamber is by orienting edges of the Coxeter diagram, i.e. taking a Dynkin quiver (see Section \ref{sec:orientation_diagram}).  The bi-coloring here is equivalent. In the Coxeter diagram, vertices are the generating reflections, thus correspond to edges in our setting. Given a non extremity vertex $x$, the two edges adjacent to $x$ correspond to two vertices in the Coxeter diagram, linked by an edge. Then $x$ is in $A$ if and only if this edge points towards extremity $a$.
    \end{remark}
    In other words, $c$ is defined by the following rules:
    $$\begin{array}{rrcl}
        \text{ for } 1 \le i < k: & c(a_i) & = & a_{i+1}\\
        \text{ for } i = k: & c(a_k) & = & b_1\\
        \text{ for } 1 \le i < n+1-k: & c(b_i) & = & b_{i+1}\\
        \text{ for } i = n+1-k: & c(b_{n+1-k}) & = & a_1\\
    \end{array}$$
    \begin{remark}\label{rem:bi_coloring_from_coxeter}
        It has the property that $A = \{a, c(a), \ldots c^{k-1}(a)\}$, $b = c^k(a)$, and $B = \{b, c(b), \ldots, c^{n-k}(b)\}$ so the bi-coloring can be recovered from the Coxeter element.
    \end{remark}
    An easy way to visualize this Coxeter element is to write the path in zigzag, with all vertices of $A$ on the bottom row, and all vertices of $B$ on the top row, and then read the Coxeter element in counter-clockwise direction. For example, if element of $\red{A}$ are \red{red} and those of $\blue{B}$ are \blue{blue}:
    $$\begin{tikzcd}
        & {\blue{b_4}} & {\blue{b_3}} && {\blue{b_2}} && {\blue{b}} \\
        {\red{a}} &&& {\red{a_2}} && {\red{a_3}}
        \arrow["{t_2}"', shift right, no head, from=1-2, to=1-3]
        \arrow[shift right=2, curve={height=6pt}, maps to, from=1-2, to=2-1]
        \arrow[shift right, maps to, from=1-3, to=1-2]
        \arrow["{t_3}", no head, from=1-3, to=2-4]
        \arrow[maps to, from=1-5, to=1-3]
        \arrow["{t_5}", no head, from=1-5, to=2-6]
        \arrow[maps to, from=1-7, to=1-5]
        \arrow["{t_1}"', no head, from=2-1, to=1-2]
        \arrow[maps to, from=2-1, to=2-4]
        \arrow["{t_4}", no head, from=2-4, to=1-5]
        \arrow[maps to, from=2-4, to=2-6]
        \arrow["{t_6}", no head, from=2-6, to=1-7]
        \arrow[shift right=2, curve={height=6pt}, maps to, from=2-6, to=1-7]
    \end{tikzcd}$$
    Now, the key point is that we can associate a bi-coloring to any ordered cam-set in a way that the Coxeter element associated to the bi-colored cam-set is the Coxeter element of the ordered cam-set.\\
    \begin{definition}
        Let $(S, \prec)$ be an ordered cam-set. Write the path $S$ as follows: 
        $$\begin{tikzpicture}[every node/.style={font=\large}]
        \node (a) at (0,0) {$a$};
        \node (x1) at (2,0) {$x_1$};
        \node (xn1) at (6,0) {$x_{n-1}$};
        \node (b) at (8,0) {$b$};

        \draw (a) -- (x1) node[midway, above] {$t_1$};
        \draw (xn1) -- (b) node[midway, above] {$t_n$};

        \draw (x1) -- (3.2,0);
        \node at (4,0) {$\cdots$};
        \draw (4.8,0) -- (xn1);
        \end{tikzpicture}$$
        Let $a$ and $b$ be the endpoints of $S$. Let $(t_1, t_2, \ldots, t_n)$ be the sequence of the edges of $S$ when we traverse the path from $a$ to $b$. For all $1\le j \le n-1$, let $x_j$ be the common vertex of $t_j$ and $t_{j+1}$. Let $A_{\prec} = \{a\} \cup \{x_j | t_j \prec t_{j+1}\}$ and $B_{\prec} = \{b\} \cup \{x_j | t_{j+1} \prec t_j\}$. Then $(S, \{A_{\prec}, B_{\prec}\})$ is the bi-colored cam-set associated with $(S, \prec)$.
    \end{definition}
    \begin{proposition}
        Assume $n > 0$. Let $(S, \prec)$ be an ordered cam-set, and $(S, \{A_{\prec}, B_{\prec}\})$ be the associated bi-colored cam-set. Then $c(S, \prec) = c(S, \{A_{\prec}, B_{\prec}\})$.
    \end{proposition}
    \begin{proof}
        We take notations from the previous definitions, and we set $\red{A = A_{\prec}}$ and $\blue{B = B_{\prec}}$. Let $c = c(S, \prec)$. The proof is by induction on $n$. If $n=1$, then $S = \{(a\ b)\}$, and $c(S, \prec) = (a\ b) = c(S, \{A, B\})$.\\
        Let $n \ge 2$. If we remove $t_n$ from $S$, the remaining edges form a path $S' = a \longline x_1 \longline \cdots \longline b'$, with $x_{n-1}=b'$, and $(S', \prec)$ with $\prec$ restricted to $S'$ is an ordered cam-set with $n-1$ edges. Let $(S', \{A', B'\})$ be the associated bi-colored cam-set. By induction hypothesis, we have $c' = c(S', \prec) = c(S', \{A', B'\})$. Now, $c=c(S, \prec)$ can be written as $c=c_1t_nc_2$. Moreover, the product $c_1c_2$ is the product of the transpositions in $S'$ in the order given by $\prec$, thus $c_1c_2 = c'$. Write $\red{a_1, \ldots, a_k}$ the natural ordering of $\red{A'}$ and $\purple{b' = b_1}, \blue{b_2, \ldots, b_{n-k}}$ the natural ordering of $\blue{B'}$. Then we have $c' = (\red{a_1\ \ldots\ a_k}\purple{\ b'}\blue{\ b_2\ \ldots\ b_{n-k}})$ and $t_n = (\purple{b'}\blue{\ b})$. We have now two cases depending on if $\red{b'\in A}$ or $\blue{b' \in B}$.\\
        \[\begin{tikzcd}
            & {\blue{b_4}} & {\blue{b_3}} && {\blue{b_2}} && {\purple{b'}} & {\blue{b}} \\
            {\red{a}} &&& {\red{a_2}} && {\red{a_3}} \\
            {}
            \arrow["{t_2}"', shift right, no head, from=1-2, to=1-3]
            \arrow[shift right=2, curve={height=6pt}, maps to, from=1-2, to=2-1]
            \arrow[shift right, maps to, from=1-3, to=1-2]
            \arrow["{t_3}", no head, from=1-3, to=2-4]
            \arrow[maps to, from=1-5, to=1-3]
            \arrow["{t_5}", no head, from=1-5, to=2-6]
            \arrow[maps to, from=1-7, to=1-5]
            \arrow["{t_7}"', dashed, no head, from=1-7, to=1-8]
            \arrow["{t_1}"', no head, from=2-1, to=1-2]
            \arrow[maps to, from=2-1, to=2-4]
            \arrow["{t_4}", no head, from=2-4, to=1-5]
            \arrow[maps to, from=2-4, to=2-6]
            \arrow["{t_6}", no head, from=2-6, to=1-7]
            \arrow[shift right=2, curve={height=6pt}, maps to, from=2-6, to=1-7]
        \end{tikzcd}\]
        If $\red{b' \in A}$, then $t_{n-1} \prec t_n$, so $t_{n-1}$ appears in $c_1$ and not in $c_2$, thus $t_n$ commutes with $c_2$, and we have $c = c't_n$. Moreover, the bi-coloring of $S$ is obtained in this case by: $\red{A = A'\cup \{b'\}}$ and $\blue{B=B'\setminus \{b'\} \cup \{b\}}$, so $c = c't_n = (\red{a_1\ \ldots\ a_k\ b'}\blue{\ b\ b_2\ \ldots\ b_{n-k}})$ which is $c(S, \{\red{A}, \blue{B}\})$ because $\red{b'}$ is the last element of $\red{A}$.\\
        If $\blue{b' \in B}$, then $t_n \prec t_{n-1}$ so $t_{n-1}$ appears in $c_2$ and not in $c_1$, thus $t_n$ commutes with $c_1$, and we have $c = t_nc'$. Moreover, the bi-coloring of $S$ is obtained in this case by: $\red{A = A'}$ and $\blue{B = B' \cup \{b', b\}}$, so $c = t_nc' = (\red{a_1\ \ldots\ a_k}\blue{\ b\ b'\ b_2\ \ldots\ b_{n-k}})$ which is $c(S, \{\red{A}, \blue{B}\})$ because $\blue{b'}$ is the second element of $\blue{B}$ after $\blue{b}$.
    \end{proof}
    This result shows that for an ordered cam-set $(S, \prec)$, the Coxeter element $c(S, \prec)$ only depends on the sets $A_{\prec}$ and $B_{\prec}$. In particular, if we fix a cam-set $S = \{t_1, \ldots , t_n\}$, where $(t_i)$ is the sequence of edge from $a$ to $b$, then any ordering $\prec$ of $S$ written $s_1\prec \ldots \prec s_n$ corresponds to a permutation $\sigma \in \Sym_n$ such that $t_i = s_{\sigma(i)}$. Then, the Coxeter element only depends on the set $\{i | \sigma(i) < \sigma(i+1)\}$ of indexes $i$ such that $x_i \in A$. On one hand, to see which bi-coloring are possible, we need to compute the image set of the map $\phi:\sigma \mapsto \{i | \sigma(i) < \sigma(i+1)\}$. It is surjective, so all bi-coloring can be obtained from an ordering :
    \begin{proposition}
        Let $\phi: \sigma \in \Sym_n \mapsto \{i | \sigma(i) < \sigma(i+1)\} \in \mathcal{P}([1,n-1])$. Then $\phi$ is surjective.
    \end{proposition}
    \begin{proof}
        We prove the result by induction on $n$. If $n=1$, there is nothing to prove. Let $n \ge 2$ and let $I \subset [1,n-1]$.
        By induction hypothesis, there exists $\sigma' \in \Sym_{n-1}$ such that $\{i \in [1,n-2]  | \sigma'(i) < \sigma'(i+1)\} = I \cap [1,n-2]$. We construct $\sigma \in \Sym_n$ from $\sigma'$ as follows. If $n-1 \in I$, $\sigma$ must be increasing at the end. We set $\sigma(n) = n$ and for all $1 \le i \le n-1$, we set $\sigma(i) = \sigma'(i)$. If $n-1 \notin I$, $\sigma$ must be decreasing at the end, so we set $\sigma(n) = 1$ and for all $1 \le i \le n-1$, we set $\sigma(i) = \sigma'(i)+1$. In both cases, we have $\phi(\sigma) = I$.
    \end{proof}  
    On the other hand, distinct bi-colorings give distinct Coxeter elements by Remark \ref{rem:bi_coloring_from_coxeter}, so a cam-set can give rise to exactly $2^{n-1}$ distinct Coxeter elements.\\
    To summarize what we have done so far, we can write a \emph{weighted adjacency matrix} $M$ with columns indexed by cam-sets and rows indexed by Coxeter elements, and $M_{c,S}$ is the number of orderings $(S, \prec)$ such that $c(S, \prec) = c$. We have shown that each column has exactly $2^{n-1}$ non-zero entries, with sum $n!$ corresponding to the fiber of the map $\phi$ above. Since Coxeter element are all conjugate, each row must be identical up to permutation of the columns, so each row has the same sum and the same number of non-zero entries. Let's count these numbers and check that the total sum, and total number of non-zero entries match.\\
    If we fix a Coxeter element $c = (x_0\ x_1\ \ldots\ x_n)$, we can ask how many cam-sets $(S, \prec)$ give rise to $c(S, \prec) = c$. Such a cam-set must be bi-colored in a way that splits the sequence (or a rotation of it) into two parts $A$ and $B$. This is equivalent to choosing $a$ and $b$ in the cycle, and writing $c = (a\ a_2\ \ldots\ a_k\ b\ \ldots b_{n+1-k})$, so there is $n(n+1)/2$ such bi-colorings ($a$ and $b$ can be inverted). After choosing $A$ and $B$, the choice of a cam-set corresponds to choosing a path from $a$ to $b$ where $a, a_1, \ldots a_k$ appears in this order and $b_1, \ldots, b_{n+1-k}$ appear in the reverse order. This is equivalent to choosing the $k-1$ intermediate vertices in $A$ among the $n-1$ vertices different from $a$ and $b$. Thus, there are $\binom{n-1}{k-1}$ such cam-sets for a fixed choice of $A$ and $B$. To count the total number of cam-sets giving rise to $c$, we can sum, first on the choice of $a$, then on $k$ the size of $A$, and divide by $2$ to account for the inversion of $a$ and $b$: $1/2\times\sum_{a=1}^{n+1} \sum_{k=1}^n \binom{n-1}{k-1} = (n+1)\times 2^{n-2}$.\\
    We can check that if we multiply the number of cam-sets that can generate a fixed Coxeter element by the number of distinct Coxeter elements, we get $n+1 \times 2^{n-2} \times n! = (n+1)!\times 2^{n-2}$, and it is the same result if we multiply the number of Coxeter elements that can be generated by a fixed cam-set by the number of cam-sets, which is $2^{n-1} \times (n+1)!/2 = (n+1)!\times 2^{n-2}$, so the two counts of the number of non-zero entries in the matrix $M$ match.\\
    For the sum, we have $n!$ orderings per cam-set, and $(n+1)!/2$ cam-sets, so the total sum is $n! \times (n+1)!/2$. On the other hand, we have $n!$ Coxeter elements, so we should have $(n+1)!/2$ ordered cam-sets generating a fixed Coxeter element.

    The link with the previous section is that the first rule $(a)$ does not change the cam-set, nor the sets $A_{\prec}$ and $B_{\prec}$, while the second rule $(b)$ has the effect of swapping the two vertices of $s_n$ in the graph but almost preserving the bi-coloring, which is coherent with the fact that the Coxeter element is preserved. We will write, as in Section \ref{sec:hurwitz_on_dec}, $(S, \prec) \toa (S', \prec')$ if $S=S'$ and $\prec'$ is obtained from $\prec$ by swapping the order of two consecutive transposition that commutes, i.e. two edges with no common endpoint. Relation $(S, \prec) \tob (S', \prec')$ holds if $S'={}^{s_n}S$ where $s_n$ is the maximum of $S$ for $\prec$, and $\prec'$ is obtained by conjugating $\prec$ except for $s_n$ that becomes the minimum for $\prec'$. The reflexive transitive closure of $\toa$ (resp. $\tob$) will is denoted $\toas$ (resp. $\tobs$) and the reflexive transitive closure of the union $\eqab$.
    The goal of the following is to prove:
    \begin{theorem}\label{prop:hurwitz_cam-sets}
        Let $(S, \prec)$ and $(S', \prec')$ be two ordered cam-sets such that $c(S, \prec) = c(S', \prec')$.\\
        Then $(S, \prec) \eqab (S', \prec')$.
    \end{theorem}
    Since $\eqab$ is an equivalence relation, we need to show that it has only one class per Coxeter element.\\
    We start with this lemma that says that rule $(a)$ allows going from any ordering of a cam-set to any other ordering with the same bi-coloring:
    \begin{lemma}\label{lem:hurwitz_a_rule}
    Let $(S, \prec)$ and $(S, \prec')$ be two orderings of the same cam-set such that $(S, \{A_{\prec}, B_{\prec}\}) = (S, \{A_{\prec'}, B_{\prec'}\})$. Then $(S, \prec) \toas (S, \prec')$.
    \end{lemma}
    \begin{proof}
        Write $S$ as $a=x_0\longline\cdots\longline x_n = b$ and let $t_i$ be the edge $x_{i-1}\longline x_i$ for all $1\le i le n$. On the other hand, with order $\prec$, $S$ can be written $\{s_1 \prec\ldots \prec s_n\}$ the ordering $\prec$ and $S= \{s'_1 \prec' \ldots \prec' s'_n\}$ with the ordering $\prec'$. Let $N$ be the largest index such that for all $1 \le i \le N$, $s_i = s'_i$. We prove the result by induction on $n-N$. If $N=n$, then $\prec = \prec'$ and we are done.\\
        Assume $N < n$. Then $s'_{N+1} = s_i = t_j$ for some $i > N+1$ and some $1 \le j \le n$. We want to move $t_j$ to position $N+1$ by applying rule $(a)$ several times. For this, we need to prove that all of $s_{N+1}, \ldots s_{i-1}$ commutes with $t_j$, i.e. none of these are $t_{j-1}$ or $t_{j+1}$. Assume by contradiction that $t_{j-1} = s_l$ for some $N < l < i$. Then $t_{j-1} \prec t_j$ so $x_{j-1} \in A_{\prec} = A_{\prec'}$ so $t_{j-1} \prec' t_j=s'_{N+1}$. Thus, $t_{j-1} = s'_m$ for some $m \le N$, but the $N$ smallest elements of $\prec'$ and $\prec$ are the same, so $t_{j-1} = s_m = s_l$ but this is impossible since $m \le N < l$. Similarly, if $t_{j+1} = s_l$ for some $N < l < i$, then $x_{j+1} \in B_{\prec} = B_{\prec'}$ so $t_{j+1} \prec' t_j=s'_{N+1}$, thus $t_{j+1} = s'_m$ for some $m \le N$, and we get the same contradiction. Thus, neither $t_{j-1}$ nor $t_{j+1}$ appears between $s_{N+1}$ and $s_i$, so we can move $t_j$ to position $N+1$ by applying rule $(a)$ several times : $(S, \prec) \toas (s_1,\ldots,s_N, t_j, s_{N+1}, \ldots, s_{i-1}, s_{i+1}, \ldots, s_n)$. We conclude by induction hypothesis.\\
    \end{proof}
    The previous result shows that the relation $\toas$ is total inside the set of orderings that give the same bi-colored cam-set. We will write $(S, \{A, B\}) \tobs (S', \{A', B'\})$ if there exists orderings $\prec$ and $\prec'$ such that $(S, \prec) \tobs (S', \prec')$ and $(S, \{A, B\})$ (resp. $(S', \{A', B'\})$) is the bi-colored cam-set associated with $(S, \prec)$ (resp. $(S', \prec')$).\\
    Now, let's show that if we reached the same endpoints, then we got the same bi-coloring:
    \begin{lemma}\label{lem:same_endpoints}
        Let $(S, \prec)$ and $(S', \prec')$ be two ordered cam-sets with the same endpoints $a$ and $b$, such that $c(S, \prec) = c(S', \prec')$. Then $\{A_{\prec}, B_{\prec}\} = \{A_{\prec'}, B_{\prec'}\}$.
    \end{lemma}
    \begin{proof}
        Let $a \in A_{\prec}$ and $b \in B_{\prec}$ be the two endpoints of the path $S$. Up to inverting $A_{\prec'}$ and $B_{\prec'}$, we can assume that $a \in A_{\prec'}$ and $b \in B_{\prec'}$. Let $c = c(S, \prec) = c(S', \prec')$. There is some $k$ such that $c^k(a) = b$. Then, we have $A_{\prec} = \{a, c(a), c^2(a), \ldots, c^{k-1}(a)\} = A_{\prec'}$ and $B_{\prec} = \{b, c(b), c^2(b), \ldots, c^{n-k}(b)\} = B_{\prec'}$.
    \end{proof}
    The action of rule $(b)$ on the cam-set is to conjugate it by the last reflection, which has the effect of swapping the two vertices of this reflection in the graph by Proposition \ref{prop:cam-set_conjugation}. The action on the bi-coloring differs depending on whether one of these vertices is an endpoint of the cam-set.
    \begin{lemma}\label{lem:hurwitz_b_rule}
        Let $(S, \prec) = (s_1, \ldots, s_n)$ be an ordered cam-set, and $(S, \{A_{\prec}, B_{\prec}\})$ be the associated bi-colored cam-set. Let $a$ (resp. $b$) be the endpoint in $A_{\prec}$ (resp. $B_{\prec}$), $t_1,  \ldots, t_n$ be the sequence of edges from $a$ to $b$, and $a=x_0, x_1, \ldots x_n = b$ the sequence of vertices from $a$ to $b$. Then $s_n = t_k$ for some $k$. Let $(S', \prec') = (s_n, {}^{s_n}s_1, \ldots, {}^{s_n}s_{n-1})$ be the ordered cam-set obtained by applying rule $(b)$ to $(S, \prec)$, and let $(S', \{A_{\prec'}, B_{\prec'}\})$ be the associated bi-colored cam-set. Then:
        \begin{itemize}
            \item If $k = 1$, then $x_1 \in B_{\prec}$, and $\{A_{\prec'}, B_{\prec'}\} = \{A_{\prec} \cup \{x_1\}, B_{\prec} \setminus \{x_1\}\}$.
            \item If $k = n$, then $x_{n-1} \in A_{\prec}$, and $\{A_{\prec'}, B_{\prec'}\} = \{A_{\prec} \setminus \{x_{n-1}\}, B_{\prec} \cup \{x_{n-1}\}\}$.
            \item If $1 < k < n$, then $\{A_{\prec'}, B_{\prec'}\} = \{A_{\prec}, B_{\prec}\}$.
        \end{itemize}
    \end{lemma}
    \begin{proof}
        Assume first that $k=1$. Then $s_n = t_1$ is the maximum of $\prec$, so $t_2 \prec t_1$, thus $x_1 \in B_{\prec}$. By Proposition \ref{prop:cam-set_conjugation}, $S'= x_1 \longline a \longline x_2 \longline \cdots \longline x_n$. Write $t'_1, \ldots, t'_n$ the sequence of edges of $S'$ from $x_1$ to $b$. Then by Proposition \ref{prop:cam-set_conjugation}, we have $t'_1 = t_1 = (x_1\ a), t'_2=(a\ x_2), t'_i = t_i$ for $i \ge 3$. By definition of rule $(b)$, $\prec'$ is defined by $t'_1 \prec' t'_i$ for all $i \ge 2$, and for $2 \le i,j \le n$, $t'_i \prec' t'_j$ if and only if $t_i \prec t_j$. Thus, if we take $A_{\prec'}$ as the part of the partition that contains $x_1$, for $j \ge 2$, we have $x_j \in A_{\prec'}$ if and only if $t'_j \prec' t'_{j+1}$, if and only if $t_j \prec t_{j+1}$, if and only if $x_j \in A_{\prec}$, and $t'_1 \prec' t'_2$ so $a \in A_{\prec'}$. Thus, we have $A_{\prec'} = A_{\prec} \cup \{x_1\}$ and $B_{\prec'} = B_{\prec} \setminus \{x_1\}$.\\
        The case $k=n$ reduces to the previous one by inverting the roles of $a$ and $b$.\\
        Assume now that $1 < k < n$. Since $s_n = t_k$ is the maximum of $\prec$, we have $t_{k-1} \prec t_k$ and $t_{k+1} \prec t_k$, thus $x_{k-1} \in A_{\prec}$ and $x_k \in B_{\prec}$. By Proposition \ref{prop:cam-set_conjugation}, we know that $S'$ that we denote $S' = x'_0 \longline \cdots \longline x'_n$ can be written 
        $$S' = a \longline x_1 \longline \cdots \longline x_{k-1} \longline x_{k+1} \longline x_k \longline x_{k+2} \longline \cdots \longline x_n$$ 
        Write $t'_1,\ldots, t'_n$ the sequence of edges of $S'$ from $a$ to $b$. Then by Proposition \ref{prop:cam-set_conjugation}, we have $t'_i = t_i$ for all $i \ne k-1, k+1$, $t'_{k-1} = (x_{k-1}\ x_{k+1})$,  $t'_{k+1} = (x_k\ x_{k+2})$. By definition of rule $(b)$, $\prec'$ is defined by $t'_k \prec' t'_i$ for all $i \ne k$, and for $i,j \ne k$, $t'_i \prec' t'_j$ if and only if $t_i \prec t_j$. Thus, if we take $A_{\prec'}$ as the part of the partition that contains $a$, for all $j \ne k-1, k$, we have $x'_j=x_j \in A_{\prec'} \Leftrightarrow t'_j \prec' t'_{j+1} \Leftrightarrow t_j \prec t_{j+1}$ if and only if $x_j \in A_{\prec}$. Let's look at $x_{k-1} = x'_k$: we have $t'_k \prec' \Leftrightarrow t'_{k+1}$, so $x_{k-1}=x'_k \in A_{\prec'}$ (remember it was already in $A_{\prec}$). Now, look at $x_k = x'_{k-1}$: we have $t'_k \prec' t'_{k-1}$, so $x_k = x'_{k-1} \in B_{\prec'}$ (remember it was already in $B_{\prec}$). Thus, we have $A_{\prec'} = A_{\prec}$ and $B_{\prec'} = B_{\prec}$.\\ 
    \end{proof}
    We now understand how rule $(b)$ acts on bi-colored cam-sets. Combining this with Lemma \ref{lem:hurwitz_a_rule}, we will be able to work on bi-colored cam-sets.
    \begin{lemma}\label{lem:hurwitz_swap}
        Let $(S, \{A, B\})$ be a bi-colored cam-set, let $a \in A$ and $b \in B$ be the two endpoints of $S$. Wright $a=x_0, x_1, \ldots, x_{n-1}, x_n = b$ be the vertices of $S$ from $a$ to $b$. Let $1 < j < n$ be such that $x_{j-1} \in A$ and $x_j \in B$, and let $S' = {}^{t_j}S$ be the cam-set obtained from $S$ by swapping the vertices $x_{j-1}$ and $x_j$. Then $(S, \{A, B\}) \tobs (S', \{A, B\})$.
    \end{lemma}
    \begin{proof}
        Let $(S, \prec) = (s_1, \ldots s_n)$ be an ordering of $S$ such that $(S, \{A, B\})$ is the associated bi-colored cam-set.\\
        Assume first that $t_j$ is the maximum of $\prec$. Then, by Lemma \ref{lem:hurwitz_b_rule}, applying rule $(b)$ to $(S, \prec)$ gives an ordering $(S', \prec')$ of $S'$ with the same associated bi-coloring, so $(S, \{A, B\}) \tobs (S', \{A, B\})$.\\
        Now, in the general case let $t_j = s_i$ for some $i$. Since $x_{j-1} \in A$ and $x_j \in B$, then $t_{j-1} \prec t_j$ and $t_{j+1} \prec t_j$. Thus, neither $t_{j-1}$ nor $t_{j+1}$ appears between $s_{i+1}$ and $s_n$. So, by applying rule $(a)$, we can move $t_j$ to the last position: $(S, \prec) \toas (s_1, \ldots, s_{i-1}, s_{i+1}, \ldots, s_n, t_j)$, which reduces to the previous case.
    \end{proof}
    The previous lemma shows that we can use rules $(a)$ and $(b)$ to swap two adjacent vertices of the cam-set that are in different parts of the bi-coloring. Since $\eqab$ is symmetric, we can then permute arbitrarily the vertices (except the endpoints) of the cam-set while preserving the bi-coloring.\\
    \begin{lemma}\label{lem:hurwitz_fixed_endpoints}
        Let $(S, \prec)$ and $(S', \prec')$ be two ordered cam-sets such that $c(S, \prec) = c(S', \prec')$. Assume that the endpoints of $S$ and $S'$ are the same, denoted by $a$ and $b$. Then $(S, \prec) \eqab (S', \prec')$.
    \end{lemma}
    \begin{proof}
        By Lemma \ref{lem:same_endpoints}, the associated bi-colorings are, $A = \{a, c(a), \ldots, c^{k-1}(a)\}$ and $B = \{b, c(b), \ldots, c^{n-k}(b)\}$ with $c^k(a) = b$. Let $S''= a \longline c(a), \dots \longline c^{k-1}(a) \longline c^{n-k}(b) \longline \cdots \longline c(b) \longline b$, and let $\prec''$ be an ordering of $S''$ with associated bi-colored cam-set $(S'', \{A, B\})$. It is enough to show that $(S, \prec) \eqab (S'', \prec'')$. Indeed, we will also have $(S', \prec') \eqab (S'', \prec'')$, and then, since $\eqab$ is symmetric and transitive, we will get $(S, \prec) \eqab (S', \prec')$.\\
        Let $x_1, \ldots, x_{n-1}$ be the vertices of $S$ between $a$ and $b$. Let $M(S) = \sum_{x_j \in A} j$ be the sum of the indices of the vertices in $A$. In $S''$, the vertices of $A$ are in positions $1, 2, \ldots, k-1$, so their sum of indices is minimal, and it is the only cam-set with this property. Then we prove by induction on $M(S)$ that $(S, \{A, B\}) \tobs (S'', \{A, B\})$. If $M(S)$ is minimal, then $S = S''$ and we are done.\\
        Assume now that $M(S)$ is not minimal. Then $S \ne S''$, so there exists some $j$ such that $x_j \in B$ and $x_{j+1} \in A$. By Lemma \ref{lem:hurwitz_swap}, $(S, \{A, B\}) \tobs (S_1, \{A, B\})$ where $S_1$ is obtained from $S$ by swapping $x_j$ and $x_{j+1}$, and $M(S_1) = M(S) - 1$. By induction hypothesis, we have $(S_1, \{A, B\}) \tobs (S'', \{A, B\})$, so we conclude that $(S, \{A, B\}) \tobs (S'', \{A, B\})$.
    \end{proof}
    The last intermediate result, and interesting by itself, is that we can use rule $(b)$ to choose arbitrarily the endpoints of the cam-set.
    \begin{lemma}\label{lem:hurwitz_swap_extremity}
        Let $(S, \{A, B\})$ be a bi-colored cam-set, let $a \in A$ and $b \in B$ be the two endpoints of $S$. Let $a=x_0, x_1, \ldots, x_{n-1}, x_n = b$ be the vertices of $S$ from $a$ to $b$. Assume that $x_1 \in B$ (resp. $x_{n-1} \in A$), and let $S' = {}^{t_1}S$ (resp. $S' = {}^{t_n}S$) be the cam-set obtained from $S$ by swapping the vertices $x_0$ and $x_1$ (resp. $x_{n-1}$ and $x_n$). Then $(S, \{A, B\}) \tobs (S', \{A \cup \{x_1\}, B \setminus \{x_1\}\})$ (resp. $(S, \{A, B\}) \tobs (S', \{A \setminus \{x_{n-1}\}, B \cup \{x_{n-1}\}\})$).
    \end{lemma}
    \begin{proof}
        Let $(S, \prec) = (s_1, \ldots s_n)$ be an ordering of $S$ such that $(S, \{A, B\})$ is the associated bi-colored cam-set.\\
        Assume first that $t_1$ (resp. $t_n$) is the maximum of $\prec$. Then, by Lemma \ref{lem:hurwitz_b_rule}, applying rule $(b)$ to $(S, \prec)$ gives an ordering $(S', \prec')$ of $S'$ with associated bi-coloring $(S', \{A \cup \{x_1\}, B \setminus \{x_1\}\})$ (resp. $(S', \{A \setminus \{x_{n-1}\}, B \cup \{x_{n-1}\}\})$), so $(S, \{A, B\}) \tobs (S', \{A \cup \{x_1\}, B \setminus \{x_1\}\})$ (resp. $(S, \{A, B\}) \tobs (S', \{A \setminus \{x_{n-1}\}, B \cup \{x_{n-1}\}\})$).\\
        Now, in the general case let $t_1 = s_i$  (resp. $t_n = s_i$) for some $i$. Since $x_1 \in B$ (resp. $x_{n-1} \in A$), then $t_2 \prec t_1$ (resp. $t_{n-1} \prec t_n$). Thus, $t_2$ (resp. $t_{n-1}$) does not appear between $s_{i+1}$ and $s_n$. So, by applying rule $(a)$, we can move $t_1$ (resp. $t_n$) to the last position: $(S, \prec) \toas (s_1, \ldots, s_{i-1}, s_{i+1}, \ldots, s_n, t_1)$ (resp. $(S, \prec) \toas (s_1, \ldots, s_{i-1}, s_{i+1}, \ldots, s_n, t_n)$), which reduces to the previous case.
    \end{proof}
    The following lemma shows how to change the endpoints using Hurwitz move.
    \begin{lemma}\label{lem:hurwitz_change_endpoints}
        Let $(S, \{A, B\})$ be a bi-colored cam-set, and let $a \in A$ and $b \in B$ be the two endpoints of $S$, and $a', b'$ be two distinct vertices. Then there exists a bi-colored cam-set $(S', \{A', B'\})$ with endpoints $a'$ and $b'$, such that $(S, \{A, B\}) \tobs (S', \{A', B'\})$.
    \end{lemma}
    \begin{proof}
        Let $c = c(S, \{A, B\})$ be the Coxeter element associated to the bi-colored cam-set. Recall that applying rule $(b)$ $n$ times has the effect of conjugating the cam-set by the Coxeter element $c$. The effect on the cam-set is to apply $c$ to all vertices, making $c(a)$ and $c(b)$ the new endpoints. Since $c$ is an $(n+1)$-cycle, $a'=c^p(a)$ for some $p$, then by applying rule $(b)$ $n\times p$ times, we can reach a bi-colored cam-set: $(S, \{A, B\}) \tobs (S_1, \{A_1, B_1\})$ with $S_1 = {}^{c^p}S$ so $a'$ is an endpoint of $S_1$. Then it is enough to prove the lemma in the case where $a'=a$.\\
        Let $1 \le k < n$ be such that $c^k(a) = b$, ($k = |A|$), and $k'$ such that $c^{k'}(a) = b'$. If $k = k'$, then $b=b'$ and we are done. We prove the result by induction on $|k-k'|$. If $|k-k'|=0$, then we are done.\\
        Assume now that $|k-k'| > 0$. Write $x_0=a, x_1, \ldots, x_n=b$ the vertices of $S$ from $a$ to $b$. We have to distinguish two cases, if $b' \in A$ or $b' \in B$, or equivalently, if $k < k'$ or $k' < k$.\\
        Assume first that $b \in A$, i.e. $k' < k$. Let $x_j = c^{k-1}(a)$ be the last element of $A$ (distinct from $a$ since $b'\ne a$ is in $A$). Then $x_i \in B$ for all $i > j$. Thus, by Lemma \ref{lem:hurwitz_swap}, we can swap $x_j$ successively with the vertices $x_{j+1}, x_{j+2}, \ldots, x_{n-1}$, i.e. $(S, \{A, B\}) \tobs (S_2, \{A, B\})$ where $$S_2= a \longline x_1, \dots \longline x_{j-1} \longline x_{j+1} \longline x_{j+2} \longline \cdots \longline x_{n-1} \longline x_j \longline b.$$ Now, $x_j$ is in position $n-1$, and since $x_j \in A$, by Lemma \ref{lem:hurwitz_swap_extremity}, we have $(S_2, \{A, B\}) \tobs (S_3, \{A \setminus \{x_j\}, B \cup \{x_j\}\})$ where $$S_3= a \longline x_1 \longline \cdots \longline x_{j-1} \longline x_{j+1} \longline x_{j+2} \longline \cdots \longline x_{n-1} \longline b \longline x.j$$ In $S_3$, the endpoints are $a$ and $c^{k-1}(a)$, so by induction hypothesis, there exists a bi-colored cam-set $(S', \{A', B'\})$ with endpoints $a$ and $b'$, such that $(S_3, \{A \setminus \{x_j\}, B \cup \{x_j\}\}) \tobs (S', \{A', B'\})$. We conclude that $(S, \{A, B\}) \tobs (S', \{A', B'\})$.\\
        The case where $b' \in B$, i.e. $k < k'$, reduces to the previous one by inverting the roles of $A$ and $B$.
    \end{proof}
    Now, all the work is done to prove the main proposition:
    \begin{proof}[Proof of Proposition \ref{prop:hurwitz_cam-sets}]
        Let $(S, \prec)$ and $(S', \prec')$ be two ordered cam-sets such that $c(S, \prec) = c(S', \prec')$. Let $(S, \{A, B\})$ and $(S', \{A', B'\})$ be the associated bi-colored cam-sets. Let $a, b$ (resp. $a', b'$) be the endpoints of $S$ (resp. $S'$). By Lemma \ref{lem:hurwitz_change_endpoints}, there exists a bi-colored cam-set $(S_1, \{A_1, B_1\})$ with endpoints $a', b'$, such that $(S, \{A, B\}) \tobs (S_1, \{A_1, B_1\})$. Moreover, since $c(S_1, \{A_1, B_1\}) = c(S, \prec) = c(S', \prec')$, by Lemma \ref{lem:hurwitz_fixed_endpoints}, we have $(S_1, \{A_1, B_1\}) \tobs (S', \{A', B'\})$. Thus, $(S, \{A, B\}) \tobs (S', \{A', B'\})$, which in other words means that there exists orderings $\tilde{\prec}$ and $\tilde{\prec}'$ such that $(S, \tilde{\prec}) \tobs (S', \tilde{\prec}')$ with associated bi-colored cam-sets $(S, \{A, B\})$ and $(S', \{A', B'\})$. Finally, by Lemma \ref{lem:hurwitz_a_rule}, we have $(S, \prec) \toas (S, \tilde{\prec})$ and $(S', \tilde{\prec}') \toas (S', \prec')$, so we conclude that $(S, \prec) \eqab (S', \prec')$.
    \end{proof}

\section{Results in general spherical type}
    We fix $W$ a finite irreducible geometric reflection group of rank $n$. Let $T$ be the set of all reflections of $W$. We start by investigating question \ref{question:how_many} in general type, then we will prove Theorem \ref{thm:always_eqab} in general type.\\
    \subsection{Weighted adjacency matrix}
        In order to count the number of cam-sets with which a cameral Coxeter element can be written, we introduce, like we did in type $A$, the \emph{weighted adjacency matrix}
        \begin{definition}
            Let $W$ be a finite irreducible geometric reflection group of rank $n$. The \emph{weighted adjacency matrix} is the rectangular matrix $M$ with columns indexed by cam-sets and rows indexed by cameral Coxeter elements, such that $M_{c,S}$ is the cardinal of $\{\sigma \in \Sym_n | s_{\sigma(1)}\ldots s_{\sigma(n)} = c\}$ where $S=\{s_1,...,s_n\}$, or equivalently, the cardinal of $\{(t_1,\ldots, t_n) \in \Red^c_T(c) | \{t_1, \ldots t_n\} = S\}$.
        \end{definition}
        The answer to question \ref{question:how_many} is contained in the matrix $M$. Given a cameral Coxeter element $c$, the number of cam-sets with which it can be written is the number of non-zero entries on row $c$. Since all cameral Coxeter elements are conjugated, this number does not depend on the choice of $c$, and can be computed by dividing the total number of non-zero entries in $M$ by the number of rows $\C^c(W)$. On the other hand, each column corresponds to a cam-set, and they are all conjugated, so they have the same number of non-zero entries. Thus, the total number of non-zero entries is equal to the number of columns, i.e. the number of cam-sets $|W|/2$, multiplied by the number of non-zero entries in one column, that is, the number of simple Coxeter elements with respect to a cam-set $S$ : $|\C(W,S)|$. Let $S=\{s_1,\ldots ,s_n\}$ be a fixed cam-set. Then the entries on column $S$ of the adjacency matrix correspond to the fibers of the map $\phi:\Sym_n \to \C^c(W), \sigma \mapsto s_{\sigma(1)}\ldots s_\sigma(n)$, and we want to know the number of non-empty fibers. Note that for investigating question \ref{question:how_many}, we are only interested in non-zero entries. The values of the entries will be useful for investigating question \ref{question:hur_transitive} in Lemma \ref{lem:count_total_M} where we count the sum of the entries. Finding the number of distinct Coxeter elements that can be obtained from a generating set $S$ is a problem of classic Coxeter theory, that we present now.

    \subsection{Orientations of the Coxeter Diagram}\label{sec:orientation_diagram}
        We fix a cameral generating set $S_0=\{s_1, \ldots s_n\}$, associated with a chamber $C_0$. Let $\Gamma_0$ be the graph with vertex set $S_0$ and with edges $s,t$ for all $s\ne t \in S_0$ such that $st\ne ts$ (we don't need to carry the data of the label on the edges here). In order to count the simple Coxeter element for $S_0$, we will use the fact that there is a one-to-one correspondence between distinct Coxeter elements of $\C(W, S_0)$, and orientations of the Coxeter diagram $\Gamma_0$. This result is attributed to \cite{erikssonChipfiringCoxeterElements1994} and was first published in \cite{shiEnumerationCoxeterElements1997}. We define the notion of acyclic orientation of a graph.
        \begin{definition}
            An \emph{orientation} of a graph $\Gamma = (V, E)$ is a choice of direction for each edge $\{u, v\} \in E$, i.e. a directed graph (or \emph{quiver}) $(V,O)$ where $O$ is a set containing one oriented pair $(u, v)$ (also denoted $u\to v$) for each edge $\{u,v\} \in E$. It is called \emph{acyclic} when $(V,O)$ has no directed cycle. Note that if $\Gamma$ is a tree or even a forest i.e. has no cycle, then any orientation is acyclic. A \emph{sink} (resp. a \emph{source}) is a vertex that has no outcoming (resp. incoming) edges. The set of acyclic orientations of $\Gamma$ is denoted $\Acyc(\Gamma)$.
        \end{definition}
        The results of \cite{shiEnumerationCoxeterElements1997} holds for any type (even non-spherical), and can be stated as the following :
        \begin{theorem}[\cite{shiEnumerationCoxeterElements1997}, Theorem 1.5]
            There exists a bijection between $\C(W, S_0)$ and $\Acyc(\Gamma_0)$.
        \end{theorem}
        Precisely, there are two maps $f$ and $g$ defined on the set of orderings of $S_0$. The first one, $f$, sends the ordering $(S_0, \prec)$ to the product in increasing order with respect to $\prec$, which belongs to $\C(W, S_0)$. The second map $g$, sends $(S_0, \prec)$ to the orientation of $\Gamma_0$ where each edge is oriented towards the highest vertex with respect to $\prec$. What happens is that both these maps are surjective, and have the same fibers.\par
        Counting these orientations gives that each column has $2^{|E|}$ non-zero entries, and we know the number of column so we get the numbers of non-zero entries in the adjacency matrix.  In order to count by row, we also need the number of rows :
        \begin{lemma}\label{lem:n_cam_sets}
            We have $|\C^c(W)|=|W|/h$
        \end{lemma}
        \begin{proof}
            Recall that $h$ is the order of any Coxeter element, called the "Coxeter number".\par
            Since all cameral Coxeter elements are conjugate, the orbit-stabilizer formula for the conjugacy action gives $|\C^c(W)| = \dfrac{|W|}{|C_W(c_0)|}$ where $C_W(c_0)$ is the centralizer of a Coxeter element $c_0$ in $W$. The centralizer of a Coxeter element is the set of its power. The proof of \cite{carterConjugacyClassesWeyl1972} Proposition 30, is stated for Weyl groups but works well for Coxeter groups. We then have $C_W(c_0) = \langle c_0 \rangle$ so $|C_W(c_0)| = h$.
        \end{proof}
        The matrix $M$ allows us to compute for how many chambers is a cameral Coxeter element simple :
        \begin{corollary}\label{coro:numbers}
            There are exactly $2^{n-1}$ simple Coxeter elements for $S_0$. For every Coxeter element $c$, there are exactly $h\times 2^{n-2}$ cam-sets for which $c$ is simple, and thus there are $h\times 2^{n-1}$ chambers from which we can obtain $c$.
        \end{corollary}
        \begin{proof}
            The Coxeter graphs of all spherical types are trees, so $\Gamma_0$ has $n-1$ edges, so it has $2^{n-1}$ orientations, so there are $2^{n-1}$ simple Coxeter elements for $S_0$. Now, we recall the definition of the weighted adjacency matrix $M$. The columns are indexed by cam-sets and the rows are indexed by Coxeter elements, and $M_{c,S}$ is the number of orderings $(S, \prec)$ such that $c(S, \prec) = c$. What we just saw is that each column has exactly $2^{n-1}$ non-zero entries. Thus, the matrix has in total $2^{n-1}\times|W|/2$ non-zero entries. We now divide by the number of rows, given by Lemma \ref{lem:n_cam_sets}. Each row, has $2^{n-1}\times(|W|/2)/|\C^c(W)|=2^{n-2}\times h$ non-zero entries. There are exactly $h\times 2^{n-2}$ cam-sets for which $c$ is simple, and thus there are $h\times 2^{n-1}$ chambers from which we can obtain $c$ since two opposite chambers give correspond to one cam-set.
        \end{proof}
        \begin{remark}
            In the case where $W$ is not irreducible, there are $2^{n-k}$ edges in $\Gamma_0$, so $2^{n-k}$ simple Coxeter elements for a cam-set $S$, where $k$ is the number of irreducible components. But there are also $2^k$ chambers for each cam-set, thus there are still $h\times 2^{n-1}$ chambers from which we can obtain $c$.
        \end{remark}
        For a fixed cam-set, multiple Coxeter elements can be obtained with the different orientations of the Dynkin diagram, but we will see that for a fixed Coxeter element, we can link the Hurwitz action and the quiver mutation.
    \subsection{Hurwitz action and quiver mutation}
        A classic operation that can be performed on the orientations of the Coxeter diagram is the switch operation or vertex mutation in the vocabulary of quivers: we chose a sink and turn it in a source by inverting all the entering edges. Starting from a $T$-decomposition of some cameral Coxeter element $c$ on a cam-set $S$, if we mutate a sink $s$, and, at the same times, we conjugate all the vertices (and thus the cam-set) by the sink vertex $s$, we get a new $T$-decomposition of $c$ on the cam-set ${}^sS$, geometrically, we changed the chamber $C$ by going through the wall $K(s)$. If $s$ was at the end of the $T$-decomposition, then we exactly performed a Hurwitz move of type $(b)$.
        \begin{example}
            Let $W$ be the Coxeter group of type $B_3$ :
            $\begin{tikzcd}
                a && b && c
                \arrow[no head, from=1-1, to=1-3]
                \arrow["4", no head, from=1-3, to=1-5]
            \end{tikzcd}$
            We can visualize the quiver mutation on the stereographic projection of the chambers as in Example \ref{ex:A3}. For type $B_3$, the intersection of the unit sphere with the chambers can be deformed as a cube, where each face is divided in $8$ regions. Starting with a chamber $C_0$ with cam-set $S_0 = \{s, t, u\}$, with $m_{st} = 3$ and $m_{tu} = 4$ and $m_{su} = 2$, and with the Coxeter element $c_0 = stu$ corresponding to the quiver orientation $s\to t \underset{4}{\to} u$. Mutating the sink $u$ gives the new cam-set ${}^uS_0 = \{s, utu, u\}$, and the oriented Dynkin quiver is now $s\to ~utu\underset{4}{\leftarrow}u= {}^u(s\to ~t\underset{4}{\leftarrow}u)$, corresponding to the $T$-decompositions $(s, u, utu)$ and $(u, s, utu)$ of the same Coxeter element $c_0$.
            
            \begin{figure}[ht]
                \centering
                \includegraphics[width=0.5\textwidth]{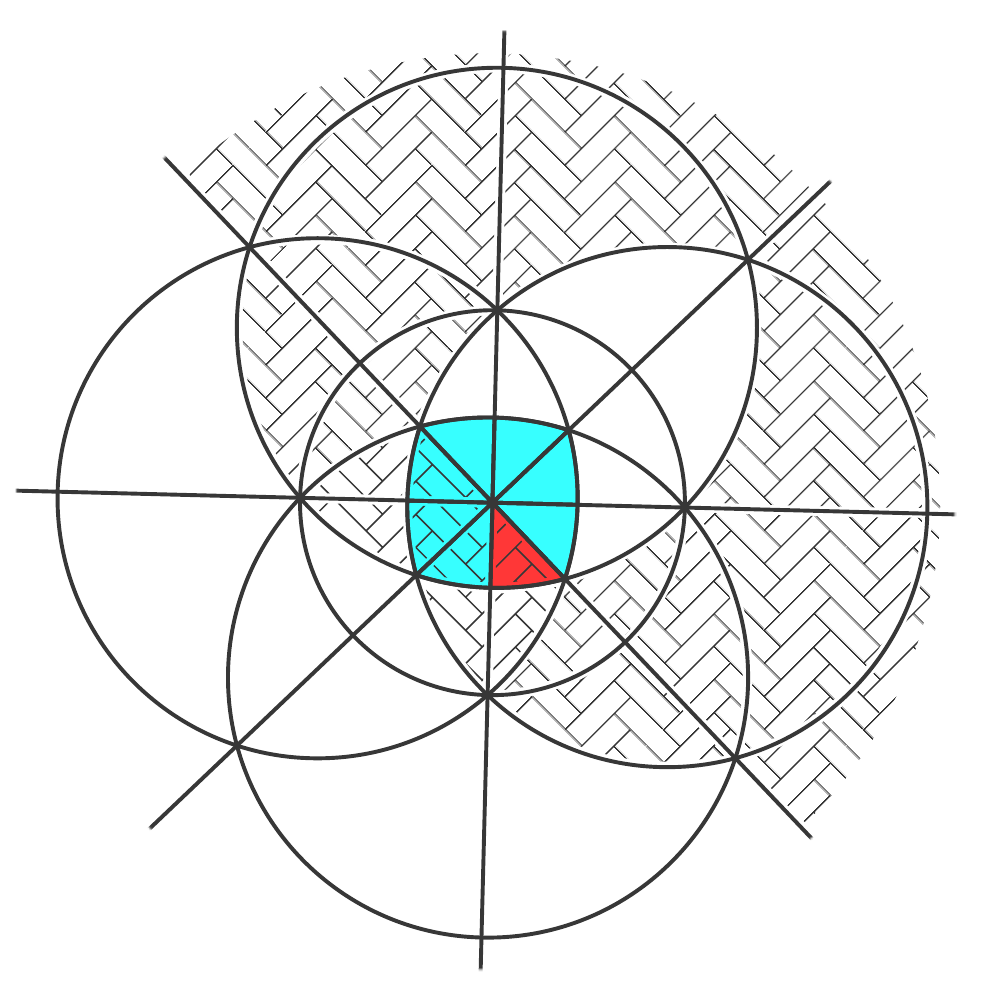}      
                \begin{tikzpicture}[scale=1.2]
                    \filldraw[fill={rgb,255:red,0;green,255;blue,255}, draw=white, line width=1.5pt] (2,2) -- (0,0) -- (2,0) -- cycle;
                    \filldraw[fill={rgb,255:red,255;green,0;blue,0}, draw=white, line width=1.5pt] (2,2) -- (2,0) -- (4,0) -- cycle;
                    \filldraw[fill={rgb,255:red,0;green,255;blue,255}, draw=white, line width=1.5pt] (2,2) -- (4,0) -- (4,2) -- cycle;
                    \filldraw[fill={rgb,255:red,0;green,255;blue,255}, draw=white, line width=1.5pt] (2,2) -- (4,2) -- (4,4) -- cycle;
                    \filldraw[fill={rgb,255:red,0;green,255;blue,255}, draw=white, line width=1.5pt] (2,2) -- (4,4) -- (2,4) -- cycle;
                    \filldraw[fill={rgb,255:red,0;green,255;blue,255}, draw=white, line width=1.5pt] (2,2) -- (2,4) -- (0,4) -- cycle;
                    \filldraw[fill={rgb,255:red,0;green,255;blue,255}, draw=white, line width=1.5pt] (2,2) -- (0,4) -- (0,2) -- cycle;
                    \filldraw[fill={rgb,255:red,0;green,255;blue,255}, draw=white, line width=1.5pt] (2,2) -- (0,2) -- (0,0) -- cycle;
                    \draw (-0.354,-0.354) -- (4.354,4.354);
                    \node at (-0.354,-0.354) [below left] {${}^ut$};
                    \node at (4.354,4.354)  [above right] {${}^ut$};

                    \draw (4.354,-0.354) -- (-0.354,4.354);
                    \node at (4.354,-0.354) [below right] {$t$};
                    \node at (-0.354,4.354) [above left]  {$t$};

                    \draw (2,-0.5) -- (2,4.5);
                    \node at (2,-0.5) [below] {$u$};
                    \node at (2,4.5)  [above] {$u$};

                    \draw (-0.5,2) -- (4.5,2);
                    \node at (-0.5,2) [left]  {${}^tu$};
                    \node at (4.5,2)  [right] {${}^tu$};

                    \draw (-0.5,0) -- (4.5,0);
                    \node at (-0.5,0) [left]  {$s$};
                    \node at (4.5,0)  [right] {$s$};

                    \draw (4,-0.5) -- (4,4.5);
                    \node at (4,-0.5) [below] {${}^ts$};
                    \node at (4,4.5)  [above] {${}^ts$};

                    \draw (4.5,4) -- (-0.5,4);
                    \node at (4.5,4)  [right] {${}^{tut}s$};
                    \node at (-0.5,4) [left]  {${}^{tut}s$};

                    \draw (0,4.5) -- (0,-0.5);
                    \node at (0,4.5)  [above] {${}^{ut}s$};
                    \node at (0,-0.5) [below] {${}^{ut}s$};

                    \node at (1.8,1.5) {$u$};
                    \node at (1.25,0.40) {$\blue{s\to t\underset{4}{\leftarrow} u}$};

                    \node at (2.2,1.5) {$1$};
                    \node at (2.75,0.40) {$\blue{s\to t\underset{4}{\to} u}$};

                    \node at (2.5,1.8) {$t$};
                    \node at (3.33,1.33) {};

                    \node at (2.6,2.25) {$tu$};
                    \node at (3.33,2.67) {};

                    \node at (2.33,2.67) {$tut$};
                    \node at (2.67,3.33) {};

                    \node at (1.67,2.67) {$tutu$};
                    \node at (1.33,3.33) {};

                    \node at (1.5,2.2) {$utu$};
                    \node at (0.75,2.4) {$\blue{s\leftarrow t\underset{4}{\leftarrow} u}$};

                    \node at (1.5,1.8) {$ut$};
                    \node at (0.75,1.4) {$\blue{s\leftarrow t\underset{4}{\to} u}$};

                \end{tikzpicture}
                \caption{The stereographic projection of the chambers in type $B_3$ and the illustration of how quiver mutations of the oriented Dynkin diagram allows traveling through the chambers in which the Coxeter element $c_0$ is simple (the chamber with the tiles), starting from $C_0$ in red. For example, in chamber ${}^{utu}C_0$, the Coxeter element $c_0$ is simple and the $T$-decomposition $({}^{utu}u, {}^{utu}t, {}^{utu}s)$ can be obtained from the mutated quiver : ${}^{utu}(\blue{s\leftarrow t \underset{4}{\leftarrow} u})$. In chamber ${}^tC_0$, the Coxeter element $c_0$ is not simple.}    
            \end{figure}
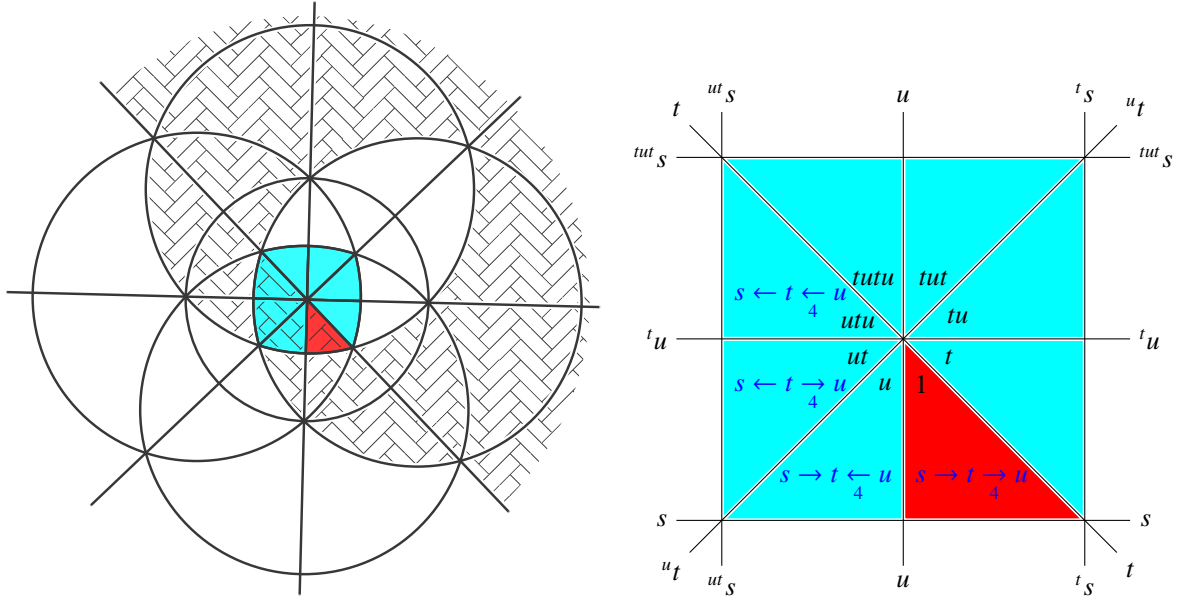
        \end{example}
        \begin{remark}
            This remark is meant as a starting point for exploration. In general, a quiver mutation of any vertex $v$ (not necessarily a sink) is performed by inverting the entering and exiting edges after adding one edge from $x$ to $y$ for every pair of edge $x\to v, v\to y$, then simplifying if there are some edges in two ways between two given vertices. The underlying unoriented graph is changed, and may contain cycles, and it is not clear what the corresponding transformation should be on the $T$-decomposition of the Coxeter element. For example, in type $A_3$, mutating $t$ in the quiver $s\longrightarrow t \longrightarrow u$ gives a cycle. On the other hand, there is a $T$-decomposition of $c=stu$ with the underlying set $S'=\{s, usu, t\}$, that is not a Coxeter generating set, (see Example \ref{exa:hurwitz_preserv_no_cox}), and if we draw the graph with vertices $S'$ and edges for non commuting vertices, we get a cycle, which is the underlying graph of the quiver obtained by mutating the orientation corresponding to $c$. These few facts may indicate that there might be a link between the graphs obtained from drawing edges between non commuting vertices of any $T$-decomposition of $c$ and the underlying graphs of all possible quivers obtained by general mutation starting from an orientation of the Coxeter diagram.
        \end{remark}
        
        In our setup, performing mutation only on sinks gives a way of traveling through different cam-sets for which the Coxeter element $c$ is simple. We might want to know if this is transitive, i.e. if any cameral decomposition of $c$ can be obtained this way. This is actually equivalent with the fact that any two cameral decomposition of the same Coxeter element are $\eqab$-equivalent as defined at the end of Section \ref{sec:hurwitz_on_dec}. Indeed, the missing piece is that quiver mutation on an orientation of the Dynkin diagram of a cam-set $S$ is allowed on any sink, while Hurwitz move of type $(b)$ on a $T$-decomposition is only allowed on the last factor. However, for any decomposition, any sink in the associated orientation commutes with all the generators on its right, so up to swapping commuting generators, one can take the sink $s$ and push it to the right of the decomposition. Then, applying the braid $\C = \braid_1 \cdots \braid_{n-1}$ conjugates the cam-set by $s$ and puts $s$ in the first position, that is, as a source. Thus, whether quiver mutation (with conjugacy by the mutated vertex) is transitive on cameral decomposition is equivalent to whether $\eqab$ has only one equivalence class per cameral Coxeter element. The positive answer is obtained in the following by counting the size of the orbit of a decomposition under these transformations.

    \subsection{Partial action on cameral decompositions}
        Our main result is the following.
        \begin{theorem}\label{thm:always_eqab}
            Let $W$ be a finite irreducible geometric reflection group of rank $n$. Let $c_0$ be a cameral Coxeter element, and let $u, v$ be two cameral decompositions in $\Red_{T}^c(c_0)$. Then $u \eqab v$.
        \end{theorem}
        Let $\Gamma$ be the Coxeter graph of $(W, S_0)$. Our strategy is to link the action of $B_n$ on ordered cam-sets with the partial action of $\Sym_n$ on ordered trees. Let $\pi$ be the projection of $B_n$ on $\Sym_n$ defined by $\pi(\braid_i) = (i\ i+1)$ (note that we also have $\pi(\C) = C = (1\ 2\ \ldots \ n)$.)
            Let $X = \bigcup_c \Red_T^c(c)$ be the set of all cameral decompositions of all cameral Coxeter elements of $W$. We restate the definition of $\toa, \tob$ in terms of partial actions. Let $R_{\bot} = \{(\braid_i, (t_1, \ldots, t_n)) \in B_n\times X| t_i \bot t_{i+1}\}$. Then $u \toa v$ if and only if there is $b \in B_n$ such that $b.u = v$ and $(b,u)\in R_{\bot}$. Similarly, let $R_\C = \{\C\}\times X$, where $\C = \braid_1 \cdots \braid_{n-1}$. Then for all $u,v\in X$ $u \tob v$ if and only if $\C.u = v$. Then for the transitive closure, let $\gamma_a$ (resp. $\gamma_b$) be the sub-partial action of $B_n$ on $X$ generated by $R_{\bot}$ (resp. $R_\C$). For all $u,v \in X$, $u \toas v$ (resp. $u \tobs v$) if and only if there exists $b \in B_n$ such that $b.u = v$ and $(b,u) \in \domain(\gamma_a)$ (resp. $(b,u) \in \domain(\gamma_b)$). Finally, let $\gamma$ be the sub-partial action of $B_n$ on $X$ generated by $R_{\bot} \cup R_\C$. Then, for all $u,v \in X$, $u \eqab v$ if and only if there exists $b \in B_n$ such that $b.u = v$ and $(b,u) \in \domain(\gamma)$. The claim of Theorem \ref{thm:always_eqab} reduces to the following.
            \begin{proposition}
                $\gamma$ is transitive on $\Red_{T}^c(c_0)$. Its orbits are exactly the sets $\Red_{T}^c(c)$ for all cameral Coxeter elements $c$ of $W$.
            \end{proposition}
            Let $u_0=(s_1, \ldots ,s_n) \in \Red_T^c(c_0)$, and let $\Omega \subset \Red_T^c(c_0)$ be its orbit under $\gamma$. We want to show that $\Omega = \Red_T^c(c_0)$, and we will prove this by bounding their cardinality. We work first on $\Omega$. Note that all cameral decomposition obtained by the partial action $\gamma$ are ordering of a conjugate cam-set to $S_0$. Given a cameral decomposition $u = (t_1, \ldots t_n)$ the set $S= \{t_1, \ldots t_n\}$ is conjugate to $S_0$ by some element $w \in W$, and a permutation $\sigma \in \Sym_n$ allows one to describe an ordering by $t_i = {}^ws_{\sigma(i)}$. But this correspondence is not one-to-one. Let $\phi : W\times \Sym_n \to X, (w, \sigma) \mapsto ({}^ws_{\sigma(1)}, {}^ws_{\sigma(2)}, \ldots, {}^ws_{\sigma(n)})$.
            \begin{lemma}\label{lem:phi_properties}
                The following properties hold:
                \begin{enumerate}[(i)]
                    \item $\phi$ is surjective
                    \item All fibers have cardinality $2$.
                    \item For all $(\braid_i, u) \in R_{\bot}$, if we write $u = \phi(w, \sigma)$, then $\braid_i.u = \phi(w,  \sigma \circ \pi(\braid_i)^{-1})$.
                    \item For all $u \in X$, if we write $u = \phi(w, \sigma)$, then $\C.u = \phi(s_{\sigma(n)}w, \sigma \circ \pi(\C)^{-1})$.
                    \item For all $u \in \Red_T^c(c)$, if we write $u = \phi(w, \sigma)$, then $\C^n.u = \phi(cw, \sigma)$. 
                \end{enumerate}
            \end{lemma}
            \begin{proof}
                \begin{enumerate}[(i)]
                    \item Let $u = (t_1, \ldots, t_n) \in X$, and let $S = \{t_1, \ldots, t_n\}$ be the associated cam-set. Since all cam-set are conjugate, there exists $w \in W$ such that $S = {}^w S_0$. Then $S = \{{}^ws_1, \ldots {}^ws_n\} = \{t_1, \ldots t_n\}$ so there is a permutation $\sigma \in \Sym_n$ such that for all $1 \le i \le n$, $t_i = {}^ws_{\sigma(i)}$, so $u = \phi(w, \sigma)$. 
                    \item Let $u = (t_1, \ldots, t_n) \in X$, and let $(w, \sigma) \in \phi^{-1}(u)$. Then $S= {}^wS_0 = {}^wS(C_0) = S(w.C_0)$. But by Remark \ref{ref:chamber_to_cam-set}, since $W$ is irreducible, there exist exactly two chambers $C_1$ and $C_2$ such that $S(C_1) = S(C_2) = S$. Since the action of $W$ on chambers is free, there is exactly one element $w_1\in W$ (resp. $w_2 \in W$) such that $C_1 = w_1.C_0$ (resp. $C_2 = w_2.C_0$). Thus, $w = w_1$ or $w = w_2$. For each case, there is a unique permutation $\sigma_1$ (resp. $\sigma_2$) such that ${}^{w_1}s_{\sigma_1(i)} = t_i$ (resp. ${}^{w_2}s_{\sigma_2(i)} = t_i$) for all $1 \le i \le n$. Thus, $\phi^{-1}(u) = \{(w_1, \sigma_1), (w_2, \sigma_2)\}$ has cardinality $2$.
                    \item Let $u = \phi(w, \sigma)$ and assume $t_i\bot t_{i+1}$. Then:\[
                    \braid_i.u = ({}^ws_{\sigma(1)}, \ldots, {}^ws_{\sigma(i-1)}, {}^ws_{\sigma(i+1)}, {}^ws_{\sigma(i)}, {}^ws_{\sigma(i+2)}, \ldots, {}^ws_{\sigma(n)}) = \phi(w, \sigma \circ (i\ i+1)).
                    \]
                    \item Let $u = \phi(w, \sigma)$. Then, $\C.u = ({}^t({}^ws_{\sigma(n)}), {}^t({}^ws_{\sigma(1)}), \ldots, {}^t({}^ws_{\sigma(n-1)}))$ where $t = {}^ws_{\sigma(n)}$. But for all $s, {}^t({}^ws) = {}^{wt}s$, and $wt=s_{\sigma(n)}w$, so we have $\C.u = ({}^{s_{\sigma(n)}w}s_{\sigma(n)}, {}^{s_{\sigma(n)}w}s_{\sigma(1)}, \ldots, {}^{s_{\sigma(n)}w}s_{\sigma(n-1)}) = \phi(s_{\sigma(n)}w, \sigma \circ (n\ n-1\ \ldots 2\ 1))$.
                    \item Applying (iv) repeatedly, we get : $\C^n.u = \phi(w', \sigma)$, with $w' = s_{\sigma(1)} s_{\sigma(2)} \cdots s_{\sigma(n)} w = c w$ since $\pi(\C) = C = (1\ 2\ \ldots\ n)$ has order $n$.

                \end{enumerate}
            \end{proof}
            \begin{remark}
                Note that in the proof of (ii), we used the irreducibility of $W$. In general, fibers of $\phi$ would have cardinality $2^k$ where $k$ is the number of irreducible components of $W$, see Remark \ref{ref:chamber_to_cam-set}.
            \end{remark}
            In order to have a lower bound for $|\Omega|$, we have the following lemma : 

            \begin{lemma}
                For all $\sigma \in \Sym_n$, there exists $w \in W$ such that $\phi(w, \sigma) \in \Omega$.
            \end{lemma}
            \begin{proof}
                The Coxeter diagram $\Gamma$ of $(W, S_0)$ has vertex set $V(\Gamma) = S_0$. Let $f_{Id} : \{1, \ldots, n\} \to V(\Gamma)$ be the orientation defined by $f_{Id}(i) = s_i$. Recall that the group $\Sym_n$ acts (totally) on the orientations $\orderings(\Gamma)$ by $\sigma.f = f \circ \sigma^{-1}$. By Lemma \ref{lem:partial_action_tree}, this action, denoted $\alpha$, is generated by the set $R_{\Gamma} = \{(f, (i\ i+1)) | f(i) \bot f(i+1)\}$ and $R_C = \{(C, f_0)\}$ where $C = (1\ 2\ \ldots \ n)$.\\
                Let $\sigma \in \Sym_n$. By Proposition \ref{prop:generated_partial_action}, there exists a sequence $g_1, \ldots, g_k$ in $\Sym_n$ and $f_1, \ldots, f_{k+1}$ in $\orderings(\Gamma)$ such that $f_1 = f_{Id}$, $\sigma^{-1} = g_k \ldots g_2 g_1$, and for all $1 \le j \le k$, $f_{j+1} = f_j \circ g_j^{-1}$ with $(f_j, g_j) \in R_{\Gamma} \cup R_C$. For each $1\le j \le k$, let $b_j \in B_n$ the braid defined by $b_j = \braid_i$ if $g_j = (i\ i+1)$, and $b_j = \C$ if $g_j = C$, such that $\pi(b_j) = g_j$. Let now $u_1 = u_0 = (s_1, \ldots, s_n)$, and define successively by application of the Hurwitz action for $1 \le j \le k$, $u_{j+1} = b_j.u_j$. Let $b = b_k \ldots b_2 b_1 \in B_n$ and $u = u_{k+1} = b.u_0$. We want to show for all $j$ the proposition $(P_j)$: $(b_j, u_j) \in R_{\bot} \cup R_\C$. In order to do so, we introduce the proposition $(P'_j)$: There exists $w_j, \sigma_j$ such that $u_j = \phi(w_j, \sigma_j^{-1})$ and $f_j = \sigma_j.f_{Id}$. Let's show that, $(P'_1)$ holds, $(P'_j)$ implies $(P_j)$ and $(P'_{j+1})$.\\
                For $j=1$, let $w_0 = e$ the identity of $W$ and $\sigma_0 = Id$ the identity of $\Sym_n$. Then $u_1 = \phi(e, Id)$, and $f_1 = f_{Id} = Id^{-1}.f_{Id}$.\\
                Let $j\ge 1$ assume $(P'_j)$ holds, and write $u_j = \phi(w_j, \sigma_j^{-1})$ with $f_j = \sigma_j.f_{Id}$. If $g_j = C$, then $b_j = \C$ and $(b_j, u_j) \in R_\C$ so $(P_j)$ holds. We get $u_{j+1} = \C.u_j = \phi(w'', \sigma_{j+1}^{-1} \circ C^{-1})$ for some $w'' \in W$, and with $\sigma_{j+1} = C \circ \sigma_j$ by Lemma \ref{lem:phi_properties} (iv). Then $\sigma_{j+1}.f_{Id} = C \circ \sigma_j.f_{Id} = g_j.f_j = f_{j+1}$, so $(P'_{j+1})$ holds.\\
                Otherwise, $(f_j, g_j) \in R_{\Gamma}$, so $g_j = (i\ i+1)$ for some $1 \le i \le n-1$, and $b_j = \braid_i$. We have $f_j(i) \bot f_j(i+1)$, thus $\sigma_j.f_{Id}(i) \bot \sigma_j.f_{Id}(i+1)$, i.e. $s_{\sigma_j^{-1}(i)} \bot s_{\sigma_j^{-1}(i+1)}$, so conjugating by $w_j$, we have ${}^{w_j}s_{\sigma_j^{-1}(i)} \bot {}^{w_j}s_{\sigma_j^{-1}(i+1)}$, i.e. $t_i \bot t_{i+1}$ if we write $u_j = (t_1, \ldots, t_n)$. Thus, $(b_j, u_j) \in R_{\bot}$ so $(P_j)$ holds. We get $u_{j+1} = \braid_i.u_j = \phi(w_j, \sigma_j^{-1} \circ (i\ i+1))$ by Lemma \ref{lem:phi_properties} (iii). Then, with $\sigma_{j+1} = (i\ i+1) \circ \sigma_j$, we have $\sigma_{j+1}.f_{Id} = (i\ i+1) \circ \sigma_j.f_{Id} = g_j.f_j = f_{j+1}$, so $(P'_{j+1})$ holds.\\
                Then $(b, u_0) \in \domain(\gamma)$ so $u = b.u_0 \in \Omega$ and $u = \phi(w_{k+1}, \sigma_{k+1}^{-1})$. Looking at the proof, we have $\sigma_{k+1} = g_k \ldots g_2 g_1 = \sigma^{-1}$, thus $u = \phi(w_{k+1}, \sigma)$.
            \end{proof}
            Now, we can conclude the proof of the theorem by counting the cardinality of $\Omega$. First a lower bound:
            \begin{lemma}\label{lem:count_total_M}
                We have $|\Omega| \ge \dfrac{n! \times h}{2}$
            \end{lemma}
            \begin{proof}
                Instead of counting directly the cardinality of $\Omega$, we count the cardinality of its pre-image by $\phi$. It is partitioned depending on the value of $\sigma \in \Sym_n$. Let $\sigma \in \Sym_n$. By the previous lemma, the set $Y_{\sigma} = \{w \in W | \phi(w, \sigma) \in \Omega\}$ is never empty. Let's show that it contains at least $h$ elements. Let $u = \phi(w, \sigma) \in \Omega$. Applying Lemma \ref{lem:phi_properties} (v), we have for all $k \in \mathbb{Z}$, $\C^{kn}.u = \phi(c^k w, \sigma) \in \Omega$, so $c^k w \in Y_{\sigma}$. But since $c$ is a Coxeter element, it has order $h$, thus the elements $c^k w$ for $0 \le k \le h-1$ are all distinct. Thus, $|Y_{\sigma}| \ge h$. Now, since $\phi^{-1}(\Omega) = \displaystyle\bigsqcup_{\sigma \in \Sym_n} Y_{\sigma} \times \{\sigma\}$, we have $|\phi^{-1}(\Omega)| = \displaystyle\sum_{\sigma \in \Sym_n} |Y_{\sigma}| \ge n! \times h$, thus $|\Omega| \ge \dfrac{n! \times h}{2}$ by Lemma \ref{lem:phi_properties} (ii).
            \end{proof}
            Now, the upper bound follows from the computation by rows and by columns of the sum of the entries of the weighted adjacency matrix.
            \begin{lemma}
                We have $|\Red_T^c(c_0)| = \dfrac{n! \times h}{2}$
            \end{lemma}
            \begin{proof}
                The set $X = \bigcup_c \Red_T^c(c)$ is partitioned into the disjoint union of the sets $\Red_{T}^c(c)$ for all cameral Coxeter elements $c$ of $W$ (each one corresponding to a row of the weighted adjacency matrix $M$). They all have the same cardinality since they are conjugate. Thus, we have $|X| = |\Red_T^c(c_0)| \times |\C^c(W)|$ (i.e. the sum of the entries in $M$) where $\C^c(W)$ the number of rows. Lemma \ref{lem:n_cam_sets} gives $|\C^c(W)|=|W|/h$. Counting now by column, we have $|X| = n! \times |W|/2$ since there are $|W|/2$ columns, one per cam-set, and the entry in one column have sum $n!$, the number of ordering of the cam-set. Then $|\Red_T^c(c_0)| = \dfrac{|X|}{|\C^c(W)|} = \dfrac{h \times n!}{2}$. 
            \end{proof}
        This matches the lower bound on $\Omega$, so we have $\Omega = \Red_T^c(c_0)$ which concludes the proof of the theorem.
    
\section*{Acknowledgments}
    This work was done during a pre-PhD internship at the IMJ-PRG, under the supervision of Yves Guiraud and Hoel Queffelec, whom I would like to thank for the discussions we had on this topic and for proofreading the manuscript.
\bibliographystyle{elsarticle-harv} 
\bibliography{bibli.bib}

\end{document}